\documentclass[11pt,a4paper,twoside]{amsart}
\usepackage{amssymb,amsmath,amsthm,graphicx,color}
\theoremstyle{plain}
\newtheorem{theorem}{Theorem}[section]
\newtheorem{proposition}[theorem]{Proposition}
\newtheorem{definition}[theorem]{Definition}
\newtheorem{lemma}[theorem]{Lemma}

\newtheorem{assumption}[theorem]{Assumption}
\theoremstyle{remark}
\newtheorem{remark}[theorem]{Remark}

\usepackage[
   bookmarks=true,
   bookmarksnumbered=false,
   bookmarkstype=toc]{hyperref}

\allowdisplaybreaks[4]
\numberwithin{equation}{section}

\newcommand{\C}{\mathbb{C}}
\newcommand{\R}{\mathbb{R}}
\newcommand{\Z}{\mathbb{Z}}

\renewcommand{\Im}{\operatorname{Im}}

\newcommand{\I}{\infty}
\newcommand{\abs}[1]{\left\lvert #1\right\rvert}
\newcommand{\Jbr}[1]{\left\langle #1 \right\rangle}
\newcommand{\norm}[1]{\left\lVert #1\right\rVert}

\newcommand{\IN}{\quad\text{in }}

\def\({\left(}
\def\){\right)}
\def\<{\left\langle}
\def\>{\right\rangle}
\def\le{\leqslant}
\def\ge{\geqslant}

\def\d{{\partial}}
\def\l{\lambda}

\newcommand{\eps}{\varepsilon}

\DeclareMathOperator{\supp}{supp}

\newcommand{\pt}{\partial}

\begin{document}

\title[Quadratic 
nonlinear Klein-Gordon equation in 2D]
{Long range scattering for the\\
complex-valued Klein-Gordon equation\\ 
with quadratic nonlinearity\\
in two dimensions}
\author[S.Masaki, J.Segata, K.Uriya]
{Satoshi Masaki, Jun-ichi Segata and Kota Uriya}

\maketitle

\begin{abstract}
In this paper, we study large time behavior of complex-valued solutions to nonlinear Klein-Gordon equation with a
gauge invariant quadratic nonlinearity in two spatial dimensions.
To find a possible asymptotic behavior, we consider the final value problem.
It turns out that one possible behavior is a linear solution with a logarithmic phase correction as in the real-valued case.
However, 
the shape of the logarithmic correction term has one more parameter which is also given by the final data.
In the real case the parameter is constant so one cannot see its effect. However, in the complex case it varies in general.
The one dimensional case is also discussed.
\end{abstract}

\section{Introduction}

This paper is devoted to the study of asymptotic behavior of solutions to nonlinear Klein-Gordon equation 
\begin{equation}\label{E:NLKG2}
(\square_{t,x} + 1)u = \l |u|u,\quad (t,x) \in \R^{1+2},
\end{equation}
where $\square_{t,x} = \partial_t^2 - \partial_{x_1}^2 - \partial_{x_2}^2$ is d'Alembertian and $\l$ is a given real constant.
The aim of this paper is to find a possible asymptotic behavior of complex-valued solutions.
For this purpose, we consider the final value problem of \eqref{E:NLKG2}, that is, we solve the equation with the condition
\[
	u-u_{\mathrm{ap}} \to 0 \IN L^2 \quad\text{ as }t\to\I
\]
for a given asymptotic profile $u_{\mathrm{ap}}(t,x)$.
If the asymptotic behavior is correctly chosen then one can find a solution to \eqref{E:NLKG2} which actually behaves like the profile for large time.

There is a number of articles on asymptotic behavior of nonlinear Klein-Gordon equation.
Especially, cubic equations in one dimension and quadratic equations in two dimensions are known to be critical in view of 
large time behavior of the solution (see \cite{G,Mat}), and hence are extensively studied.
Two dimensional quadratic case is studied in \cite{DFX,KOS,KS,OTT} with polynomial type nonlinearities.
As for the gauge invariant quadratic nonlinearity, the first two authors \cite{MS2} considered the final value problem and
derived the asymptotic profile. For other related results, see references in \cite{MS2}.

However, there are few previous results on the behavior of complex-valued solutions in the critical nonlinearity case.
For one dimensional cubic case, Sunagawa \cite{S1} derived a sharp time decay estimate.
In the present paper, we show that \eqref{E:NLKG2} admits a solution which behaves like a linear solution with a logarithmic phase correction as time goes to infinity.
Although the type of the behavior is the same as in the real-valued case,  
it will turn out that the shape of the phase correction is much complicated than in a real-valued case.

We also consider the cubic equation in one dimension for comparison.
Remark that two dimensional quadratic case has two different difficulties.
One comes from the fact that a solution is complex-valued. The other is related to the fact that the nonlinearity is not a polynomial.
By considering the 1d cubic case, we can separate these two and focus on the first difficulty.
Comparing the 1d cubic case and 2d quadratic case, the second difficulty becomes clear.

Note that the complex valued nonlinear Klein-Gordon equation arises in various fields of 
physics. For example, the nonlinear Dirac equation which describes the self interaction of a Dirac field 
can be reduced to the system of the complex valued nonlinear Klein-Gordon equations. Therefore 
we believe that our study will aid in understanding the long time behavior of solution to 
various physical models.

\subsection{Main result}

It is well-known that the asymptotic behavior of a solution to the linear Klein-Gordon equation
\begin{equation}\label{E:lKG}
	\left\{
	\begin{aligned}
	& (\square_{t,x} + 1) v = 0, \quad (t,x) \in \R^{1+d},\\
	& v(0)= \phi_0, \quad \partial_t v(0) = \phi_1
	\end{aligned}
	\right.
\end{equation}
for large $t$
is given by
\[
	v_{\mathrm{lin}}(t,x) = t^{-\frac{d}2} {\bf 1}_{\{|x|<t\}}(t,x) [A_1(\mu) e^{i\theta} + B_1(\mu)e^{-i\theta}],
\]
where $A_1$ and $B_1$ are explicitly given by Fourier transform of $\phi_0$ and $\phi_1$ (see \eqref{E:defA1} and \eqref{E:defB1}, below) and
\begin{align}\label{E:thetamu}\textstyle
	\theta&{} =- \sqrt{t^2- |x|^2}, & \mu &{} \textstyle= \frac{x}{\sqrt{t^2-|x|^2}} \in \R^d
\end{align} 
for $t\ge 1$ and $|x|<t$.

We give our asymptotic profile as a modification of $v_{\mathrm{lin}}(t,x)$.
It is hence convenient to state our assumption in terms of $A_1$ and $B_1$.

To state our main result, we introduce notations. For a vector $x \in \R^d$, we denote $\Jbr{x} = (1+|x|^2)^{1/2}$.
For a measurable set $\Omega$ and $1\le p \le \I$, $L^p(\Omega)$ stands for the usual Lebesgue space.
We simply write $L^p$ if $\Omega = \R^d$ and if there is no fear of confusion on the choice of $d$.
Let us make a definition of a function space which we work with.
\begin{definition}[Weighted Sobolev space]
For $k \in \Z_{\ge0}$ and $s\ge 0$,
we define a weighted Sobolev space $H^{k,s}$ by 
\begin{align*}
	H^{k,s}(\R^d) := {}& \{ f\in L^2(\R^d) \ |\ \norm{f}_{H^{k,s}} <\I \},\\
	\norm{f}_{H^{k,s}} :={}&  \sum_{\alpha \in (\Z_{\ge 0})^d,\,|\alpha|\le k} \norm{\Jbr{x}^s \partial^\alpha f}_{L^2(\R^d)}
\end{align*}
\end{definition}

\begin{assumption}\label{A:2D}
Let $A_1, B_1\in H^{2,2}(\R^2)$. Suppose that there exists a constant $\rho_0\ge1$ such that the identity
\begin{equation}\label{E:a2D}
	\tfrac1\rho_0 |A_1(\xi)| \le |B_1(\xi)| \le \rho_0 |A_1(\xi)|
\end{equation}
holds true for all $\xi \in \R^2$.
Suppose further that 
\begin{equation}\label{E:zetadef}
\zeta(\xi) :=
\left\{
\begin{aligned} 
&\tfrac{|B_1(\xi)|}{|A_1(\xi)|} &&(A_1(\xi) \neq 0), \\
&\lim_{z \to \xi;\; A_1(z)\neq0} \tfrac{|B_1(z)|}{|A_1(z)|} &&(A_1(\xi) = 0)
\end{aligned}
\right.
\end{equation}
is well-defined for $\xi \in \supp A_1$.
In addition, suppose that the above $\zeta(\xi)$ can be extended to a function on $\R^2$ so that 
it satisfies
$\partial_{z_j} \zeta \in L^4(\R^2) + L^\infty(\R^2)$ and
$\partial_{z_j}\partial_{z_k}\zeta \in L^2(\R^2) + L^\infty(\R^2)$ for $j,k = 1,2$,
and the inequality
\[
	\rho_0^{-1} \le \inf_{\xi \in \R^2} \zeta(\xi) \le \sup_{\xi \in \R^2} \zeta(\xi) \le \rho_0
\]
holds for the same constant $\rho_0\ge1$ as in \eqref{E:a2D}.
We abbreviate $\zeta(\xi)=|B_1(\xi)|/|A_1(\xi)|$.
\end{assumption}

\begin{remark}
It is important to note that if $\phi_0$ and $\phi_1$ are real-valued then
we can choose $\zeta\equiv1$ since $B_1(\xi)=\overline{A_1(\xi)}$ holds.
An important characterization of the complex-valued case is that $\zeta(\xi)$ may vary.
However, $\zeta\equiv1 $ does not necessarily imply $\phi_0$ and $\phi_1$ are real.
An example is $\phi_0(x) = e^{-x^2/2}$, $\phi_1(x) = -i x e^{-x^2/2}$.
\end{remark}

\begin{remark}
Under Assumption \ref{A:2D}, it is  allowed that $\zeta(\xi)$, the ratio of $|B_1(\xi)|$
to $|A_1(\xi)|$, rapidly changes in such a sense that derivative of $\zeta(\xi)$ is not bounded.
One such example is
$\zeta(\xi) = 1 + |\xi|\big|\log |\xi|\big|^{\frac{1}{4}}$ around $\xi = 0$. 
\end{remark}

We choose the asymptotic behavior $u_{\mathrm{ap}}$ as follows. For $t\ge 1$, 
\begin{equation}\label{E:uap2D}
	u_{\mathrm{ap}}(t,x) = t^{-1} {\bf 1}_{\{|x|<t\}}(t,x) [A_1(\mu) e^{i\theta + i S_A(\mu) \log t} + B_1(\mu) e^{-i\theta + i S_B(\mu) \log t}]
\end{equation}
where $\theta$ and $\mu$ are as in \eqref{E:thetamu} and
the phase modulations are
\begin{equation}\label{E:SASB2D}
	\begin{aligned}
	S_A(z) &= -\tfrac{\l}2 \Jbr{z}^{-1} L_0(\zeta(z))|A_1(z)|, \\
	S_B(z) &= \tfrac{\l}2 \Jbr{z}^{-1} L_0(1/\zeta(z))|B_1(z)|.
	\end{aligned}
\end{equation}
Here, $\zeta(z)$ is the ratio function given in Assumption \ref{A:2D} and
$L_0$ is given by
\begin{equation}\label{E:L0EKform}\textstyle
	L_0(\zeta) 
=\frac{(1+\zeta)(7+\zeta^2)}{3\pi} E(\frac{2\sqrt{\zeta}}{1+\zeta} )
	- \frac{(1+\zeta)(1-\zeta)^2}{3\pi} K(\frac{2\sqrt{\zeta}}{1+\zeta}),
\end{equation}
where $K(k)$ is the complete elliptic integral of the first kind defined by
\[\textstyle
	K(k) = \int_0^{\pi/2} (1-k^2 \sin^2 \theta)^{-\frac12} d\theta,
\]
and $E(k)$ is the complete elliptic integral of the second kind defined by
\[\textstyle
	E(k) = \int_0^{\pi/2} (1-k^2 \sin^2 \theta)^{\frac12} d\theta.
\]

\begin{theorem}\label{T:main}
Let $A_1,B_1$ be functions satisfying Assumption \ref{A:2D}.
Define the profile $u_{\mathrm{ap}}(t,x)$ as in \eqref{E:uap2D}.
Then, for any $\delta_0\in(0,1)$ and for the constant $\rho_0\ge1$ given in Assumption \ref{A:2D},
there exists $\eps_0=\eps_0(\delta_0,\rho_0)>0$ such that if
\[
	\norm{A_1}_{L^\I(\R^2)}+\norm{B_1}_{L^\I(\R^2)}\le \eps_0
\]
then there exist $T\ge1$ and a unique solution $u
\in C([T,\I);H^{1/2})$ to \eqref{E:NLKG2} such that
\[
	\norm{u(t)-u_{\mathrm{ap}}(t)}_{L^2(\R^2)} \lesssim t^{-\delta_0}
\]
for $t\ge T$. Furthermore, if the set $\{z\in\R^2 \ |\ |A_1(z)|\neq |B_1(z)|\}$ has positive measure,
then corresponding solution $u(t,x)$ is complex-valued. 
\end{theorem}

\begin{remark}
It follows from $E(1)=1$ that $L_0(1)=\frac{16}{3\pi}$. Hence, if $\phi_0$ and $\phi_1$ are real-valued, then $\zeta(z)\equiv1$ and
so $S_A(z)$ and $S_B(z)$ coincide with those in our previous study \cite{MS2} and satisfy $S_{A}+S_B \equiv0$.
Moreover, since $|L_0(\zeta)| \lesssim \Jbr{\zeta}$ (see Appendix \ref{S:L0}), we have
\[
	|S_A(z)| + |S_B(z)| \le C|\l| \Jbr{z}^{-1}(|A_1(z)|+|B_1(z)|)
\]
under Assumption \ref{A:2D}. Furthermore, we have $L_0(\zeta)=1+o(1)$ as $\zeta\downarrow0$ and $L_0(\zeta)=\frac32 \zeta + O(\zeta^{-1})$ as 
$\zeta\to\I$ (see Appendix \ref{S:L0}). These imply
\begin{align*}
	S_A(z) \sim{}& -\tfrac{\l}2 \Jbr{z}^{-1} |A_1(z)|, & 	S_B(z) \sim{}& \tfrac{3\l}4 \Jbr{z}^{-1} |A_1(z)|, 
\end{align*}
if $|B_1(z)| \ll |A_1(z)|$ and
\begin{align*}
	S_A(z) \sim{}& -\tfrac{3\l}4 \Jbr{z}^{-1} |B_1(z)|, & 	S_B(z) \sim{}& \tfrac{\l}2 \Jbr{z}^{-1} |B_1(z)|, 
\end{align*}
if $|A_1(z)| \ll |B_1(z)|$.
These formulas can be compared with those in one dimensional case (see \eqref{E:SASB1D}, below).
\end{remark}

\begin{remark}
We can slightly generalize the Assumption \ref{A:2D}. See Remark \ref{R:genAssumption} for the detail.
\end{remark}

\subsection{One dimensional cubic case}

Our method is applicable to the one dimensional cubic case
\begin{equation}\label{E:NLKG1}
(\square_{t,x} + 1)u = \l |u|^2u,\quad (t,x) \in \R^{1+1}.
\end{equation}
In order to clarify the difficulty in two dimensional by comparison, we also consider the one dimensional case.
The choice of the asymptotic behavior is simpler.
However, as far as the authors know, this is the first result which reveals the exact form of the asymptotic behavior.
The real-valued case is studied in \cite{HN09JMP} (see \cite{HN08ZAMP} for the initial value problem).

We choose the asymptotic behavior $u_{\mathrm{ap}}$ as follows. For $t\ge 1$,
\begin{equation}\label{E:uap1D}
	u_{\mathrm{ap}}(t,x) = t^{-\frac12} {\bf 1}_{\{|x|<t\}}(t,x) [A_1(\mu) e^{i\theta + i S_A(\mu) \log t} + B_1(\mu) e^{-i\theta + i S_B(\mu) \log t}]
\end{equation}
where $\theta$ and $\mu$ are given in \eqref{E:thetamu}, the phase modulations are given by
\begin{equation}\label{E:SASB1D}
\begin{aligned}
	S_A(z) &{}= -\tfrac{\l}2 \Jbr{z}^{-1} (|A_1(z)|^2+2|B_1(z)|^2),\\
	S_B(z) &{}= \tfrac{\l}2 \Jbr{z}^{-1} (2|A_1(z)|^2+|B_1(z)|^2).
\end{aligned}
\end{equation}
Our second main theorem is as follows.
\begin{theorem}
Let $\phi_0,\phi_1 \in L^2(\R)$ be complex-valued functions satisfying $A_1, B_1\in H^{2,5/2}(\R)$.
Then, there exists $\eps_0>0$ such that if
\[
	\norm{A_1}_{L^\I(\R)}+\norm{B_1}_{L^\I(\R)}\le \eps_0
\]
then there exist $T\ge1$ and a unique solution $u(t)
\in C([T,\I);L^{2})$ to \eqref{E:NLKG1} such that
\[
	\norm{u(t)-u_{\mathrm{ap}}(t)}_{L^2} \lesssim t^{-\frac5{12}}
\]
for $t\ge T$. Furthermore, if the set $\{z \in \R \ |\ |A_1(z)|\neq |B_1(z)| \}$ has positive measure then
the corresponding solution $u(t,x)$ is complex-valued.
\end{theorem}

\begin{remark}
Candy and Lindblad \cite{CL}
obtained the large time asymptotics for the $1$d cubic nonlinear Dirac equation (Thirring model) 
which can be reduced to the system of the complex valued Klein-Gordon equations 
with derivative interactions. Here, the asymptotic behavior of the solution is given by the 
solution to the linear Dirac equation with similar logarithmic phase corrections to
that of \eqref{E:uap1D}.
\end{remark}

The rest of the paper is organized as follows.
In Section \ref{S:observation}, we make an observation on the choice of the asymptotic behavior.
We give a sufficient condition for the solvability of the final value problems in terms of the asymptotic profile in Section \ref{S:FVP}.
Then, main theorems are proven in Sections \ref{S:1d} and \ref{S:2d}.

\section{Asymptotic profile with complex final states}\label{S:observation}

In this section, we consider how to choose the phase correction term of the asymptotic behavior.
In particular, let us see the difference between the real case and the complex case.

\subsection{Asymptotics for linear equation}
Let us begin with the linear Klein-Gordon equation
\[
	(\square_{t,x} + 1) u=0, \quad (t,x)\in \R^{1+d}.
\]
To obtain asymptotic behavior, we introduce the new unknowns $\phi_\pm$ by
\[\textstyle
	\phi_\pm(t) = \frac12 (u \pm i\Jbr{\nabla}^{-1} u_t).
\]
Then, $\phi_\pm$ solve the half Klein-Gordon equation
\[
	\partial_t \phi_{\pm} \pm i\Jbr{\nabla} \phi_\pm =0,
\]
respectively. 
We apply the stationary phase method to an oscillatory integral representation of $\phi_\pm (t)= e^{\mp i\Jbr{\nabla} t}\phi_\pm(0)$ for $t>0$.
Then, one finds that there is no stationary point if $|x|\ge t$,
and one stationary point $\xi=\pm \mu$, respectively, if $|x|< t$, where $\mu$ is given in \eqref{E:thetamu}.
Hence, we have the asymptotic formula
\begin{align*}
	(e^{\mp i\Jbr{\nabla}t} \phi_{\pm,0})(x)
	\sim {}& {\bf 1}_{\{|x|<t\}}(t,x) t^{-\frac{d}2} e^{\mp i\frac{d\pi}4}\Jbr{\mu}^{\frac{d+2}2} e^{\mp i \Jbr{\mu}^{-1} t} \widehat{\phi_{\pm,0}}(\pm \mu)
\end{align*}
as $t\to\I$.
To simplify the formula, we introduce
\begin{equation}\label{E:defA1}
	A_1(\mu):= e^{- i\frac{d\pi}4}\Jbr{\mu}^{\frac{d+2}2}
	\widehat{\phi_{+,0}}(\mu) = \tfrac{e^{- i\frac{d\pi}4}}2 \Jbr{\mu}^{\frac{d}2}(\Jbr{\mu}\widehat{\phi_0}(\mu)+i\widehat{\phi_1}(\mu) ),
\end{equation}
\begin{equation}\label{E:defB1}
	B_1(\mu):= e^{ i\frac{d\pi}4}\Jbr{\mu}^{\frac{d+2}2}
	\widehat{\phi_{-,0}}(-\mu) = \tfrac{e^{ i\frac{d\pi}4}}2 \Jbr{\mu}^{\frac{d}2}(\Jbr{\mu}\widehat{\phi_0}(-\mu)-i\widehat{\phi_1}(-\mu) ),
\end{equation}
and $\theta = - \Jbr{\mu}^{-1} t$ is the same one as in \eqref{E:thetamu}.
Thus, we reach to the asymptotic behavior of the linear solution
\[
	u_{\mathrm{lin}}(t,x) =  t^{-\frac{d}2} {\bf 1}_{\{|x|<t\}}(t,x)(A_1(\mu)e^{i\theta}+B_1(\mu)e^{-i\theta}).
\]

\subsection{$1d$ cubic case}
Let us move to the nonlinear equation.
We begin with the one dimensional cubic case.
Put the asymptotic behavior $u_{\mathrm{lin}}$ of the linear solution into the nonlinearity
$N(u)= \l |u|^2 u$. Then,
\[
	N( u_{\mathrm{lin}} )
	= t^{-\frac32} {\bf 1}_{\{|x|<t\}}(t,x) N(A_1 e^{i\theta} + B_1 e^{-i\theta}).
\]	
Recall that (the asymptotic behavior of) the linear solution has the term oscillates like $e^{\pm i\theta}$. 
The key point is to find a resonant part which has the same oscillation as the linear solution.
In this case, since $N(u)$ is a polynomial,
one can do so by expanding the polynomial.
Here, we use another argument which is applicable to the two dimensional case.
We use the fact that $N$ is gauge invariant and homogeneous of degree 3 to obtain
\begin{align*}
	N(A_1 e^{i\theta} + B_1 e^{-i\theta})
	&{}= e^{i\theta + i \arg A_1} N\(|A_1| +|B_1|e^{-i\alpha}e^{-2i\theta}\) \\
	&{}= |A_1|^2 A_1 e^{i\theta} N(1 +\zeta e^{-i(2\theta+\alpha)}),
\end{align*}
where $\zeta= |B_1|/|A_1|$ and $e^{i\alpha} = \frac{A_1 \overline{B_1}}{|A_1||B_1|}$.
We apply Fourier series expansion to $F$,
\[
	N(1 +\zeta e^{-i\Theta}) = \l \sum_{n=-\I}^\I L_n(\zeta) e^{in\Theta}.
\]
Here, by means of the explicit formula of $F$, we have
\[
	L_{-2}(\zeta)= \zeta^2, \quad L_{-1}(\zeta) = (\zeta^3 + 2\zeta), \quad L_0(\zeta) = (1+2\zeta^2),\quad
	L_{1} (\zeta) = \zeta,
\]
and $L_n(\zeta)=0$ for all other $n$.
This is the crucial step in deriving the asymptotic behavior.
We emphasize that these coefficients depends on $\zeta$.
This is the feature of the complex data.
Notice that $\zeta\equiv1$ in the real data case.

Let us derive the asymptotic behavior. Note that we have the relation
\[
	\zeta^{-3} L_{-n}(\zeta) =  L_{n-1} (1/\zeta).
\]
Hence, the part which has the same oscillation as $A_1e^{i\theta}$ is 
\[
	\l |A_1|^2 A_1 e^{i\theta}L_0(\zeta) = \l(|A_1|^2 + 2|B_1|^2) A_1 e^{i\theta}.
\]
and similarly the part which has the same oscillation as $B_1e^{-i\theta}$ is
\[
	\l |A_1|^2 A_1 e^{-i\theta-i\alpha}L_{-1}(\zeta)  = \l |B_1|^2 L_0(1/\zeta)B_1 e^{-i\theta}  =\l (|B_1|^2+2|A_1|^2) B_1 e^{-i\theta}.
\]
These two determine the phase correction.
Indeed, if we set
\[
	u_{\mathrm{ap}}(t,x) =  t^{-\frac{1}2} {\bf 1}_{\{|x|<t\}}(t,x)(A_1(\mu)e^{i\theta+iS_{A}(\mu)\log t}+B_1(\mu)e^{-i\theta + iS_{B}(\mu)\log t}),
\]
then a computation shows that
\begin{align*}
	(\square +1)u_{\mathrm{ap}}(t,x) \sim  t^{-\frac{3}2} {\bf 1}_{\{|x|<t\}}(t,x)& \Big((-2 \Jbr{\mu} S_A) A_1(\mu)e^{i\theta+iS_{A}(\mu)\log t}
	\\ & +(2 \Jbr{\mu} S_B)B_1(\mu)e^{-i\theta + iS_{B}(\mu)\log t}\Big).
\end{align*}
Comparing the top terms, one sees that the phase corrections are chosen as
\begin{align*}
	S_{A}(\mu) &{}:=  -\tfrac{\l}2 \Jbr{\mu}^{-1} 
	(|A_1|^2+2|B_1|^2),\\
	S_{B}(\mu) &{}:=  \tfrac{\l}2 \Jbr{\mu}^{-1} 
	(2|A_1|^2+|B_1|^2).
\end{align*}
Remark that in the real data case, we have $|A_1|\equiv|B_1|$ and so $S_{A}=-S_{B}$.

\subsection{$2d$ quadratic case}
Let us next consider the two dimensional case.
Since the nonlinearity $N(u)=\l |u| u$ is not a polynomial, the resonant part is not picked by a simple calculation.
In \cite{MS2}, Fourier series expansion is employed to provide the step.
We here use the technique.
Since $N(u)$ is gauge invariant and homogeneous of degree two, 
\begin{align*}
	N(A_1 e^{i\theta} + B_1 e^{-i\theta})&{}= e^{i\theta + i \arg A_1} N\(|A_1| +|B_1|e^{-i\alpha}e^{-2i\theta}\)\\
	&{} = |A_1|A_1 e^{i\theta} N(1 +\zeta e^{-i(2\theta+\alpha)}).
\end{align*}
We apply Fourier series expansion;
\begin{equation}\label{E:Fourier2D}
	N(1 +\zeta e^{-i\Theta})
	= \l \sum_{n=-\I}^{\I} L_n(\zeta) e^{in\Theta},
\end{equation}
where
\begin{equation}\label{E:Lndef}
	L_n(\zeta) = \frac1{2\pi} \int_{0}^{2\pi} |1+\zeta e^{-i\Theta}|(1 +\zeta e^{-i\Theta}) e^{-in\Theta}d\Theta.
\end{equation}
Due to the fact that $|1+\zeta e^{-i\Theta}|$ is smooth if $\zeta \neq 1$ but is merely continuous if $\zeta=1$,
the property of the coefficient depends on $\zeta$.
For example, $L_n=O(|n|^{-3})$ if $\zeta=1$, but $L_n=O(|n|^{-\I})$ if $\zeta\neq1$.
We also have to care about the regularity of $L_n(\zeta)$ in $\zeta$. This shows a sharp contrast 
with the real data in which case we have $\zeta\equiv1$.

One sees from \eqref{E:Lndef} that
$L_n(\zeta)$ is a real number and that the relation
\begin{equation}\label{E:reflection2D}
	\zeta^{-2} L_{-n}(\zeta) 
	=   L_{n-1}(1/\zeta)
\end{equation}
holds for any $\zeta>0$ and $n\in\Z$.
Hence, the part which has the same oscillation as $A_1e^{i\theta}$ is 
\[
	\l |A_1|A_1 e^{i\theta} L_0(\zeta) = \l |A_1| L_0(\zeta) A_1 e^{i\theta}.
\]
and the part which has the same oscillation as $B_1e^{-i\theta}$ is
\[
	\l|A_1|A_1 e^{i\theta} L_{-1}(\zeta)e^{-i(2\theta+\alpha)}
	=\l (\zeta^{-2} L_{-1}(\zeta) )|B_1|B_1  e^{-i\theta}
	=\l L_0(1/\zeta)|B_1|  B_1e^{-i\theta}.
\]
Thus, we obtain the asymptotic profile
\[
	u_{\mathrm{ap}}(t,x) =  t^{-1} {\bf 1}_{\{|x|<t\}}(t,x)(A_1(\mu)e^{i\theta+iS_{A_1}(\mu)\log t}+B_1(\mu)e^{-i\theta+iS_{B_1}(\mu)\log t})
\]
with
\begin{align*}
	S_{A}(\mu) &{}:= - \tfrac{\l}2 \Jbr{\mu}^{-1} 
	|A_1| L_0(\zeta),\\
	S_{B}(\mu) &{}:=  \tfrac{\l}2 \Jbr{\mu}^{-1} 
	|B_1| L_0(1/\zeta).
\end{align*}

\begin{remark}
In two dimensional case, the coefficient $L_0(\zeta)$ given in \eqref{E:Lndef} is described
by the elliptic integrals as in \eqref{E:L0EKform}.
We give a proof of \eqref{E:L0EKform} in Appendix \ref{S:L0}.
\end{remark}

\begin{remark}
The technique of the decomposition of the nonlinearity is first introduced in \cite{MS1}
for a study of a singular limit problem in generalized Korteweg-de Vries equation.
It is first applied to modified scattering problems for Schr\"odinger equation in \cite{MM} (see also \cite{MMU}).
In \cite{MS2,MS3}, it turns out that the decomposition is useful in the modified scattering problem for nonlinear Klein-Gordon equation.
\end{remark}

One new respect of the complex data case is that we need to care about regularity of $L_n(\zeta)$ in $\zeta$,
and decay rate of its derivatives as $n$ goes to $\pm\infty$.
Note that in the real case $\zeta \equiv 1$ and so no regularity issue is involved and all derivatives are identically zero.
From the formula \eqref{E:L0EKform}, we see that
\[
	L_0(\zeta) \in C^2(\R_+) \cap C^\I (\R_+\setminus \{1 \}). 
\]
Remark that the formula also shows $L_0(\zeta) \not\in C^3(\R_+)$.
We investigate the regularity of $L_0(\zeta)$ in Appendix \ref{S:L0}.

Let us conclude this section with the following estimate on the regularity and the decay of the coefficients $L_n(\zeta)$ for $n\ge1$.
\begin{proposition}\label{P:Lnest}
Let $L_n(\zeta)$ be as in \eqref{E:Lndef}.
Then, 
$L_n (\cdot) \in C^1(\R_+) \cap C^\I (\R_+ \setminus \{1\})$ for $n\ge1$.
Further, it holds for any $\rho_0 > 1$ that 
\begin{equation}\label{E:Lnreg}
\max_{k=0,1,2}\sup_{n \in \mathbb{Z}} \sup_{\zeta \in [\rho_0^{-1},\rho_0]\setminus\{1\}}\Jbr{n}^{3-k}|L^{(k)}_n(\zeta)| \lesssim_{\rho_0} 1.
\end{equation}
\end{proposition}
Recall that $|L_n(\zeta)|=O(|n|^{-\I})$ for each fixed $\zeta\neq1$.
On the other hand, the estimate $L_n(1)=O(|n|^{-3})$ is sharp.
Namely, there exists $c>0$ such that $|L_n(1)|\ge c\Jbr{n}^{-3}$.
The point of the above estimate is that it is uniform in $\zeta$ around $\zeta =1$.
Moreover, one sees from the regularity property of $L_0(\zeta)$ and the identity \eqref{E:reflection2D} that
the estimate \eqref{E:Lnreg} is sharp in such a sense that the supremum with respect to $\zeta \in [\rho_0^{-1},\rho_0]\setminus\{1\}$ is infinite
at least for $k=3$ and $n=0,-1$.
Proposition \ref{P:Lnest} is proved in Appendix \ref{S:Lnestpf}.

\begin{remark}\label{R:genSASB}
It is clear from the argument in this section that we do not need the support conditions on $A_1$ and $B_1$ to define the phase modifications.
It is because if $A_1(z)=0$ then the value of $S_A(z)$ has no meaning on the asymptotic profile, and the similar is true for the case $B_1(z)=0$.
Hence, in general case, one may define $S_A$ and $S_B$ by
\begin{align*}
	S_A(z) &{}=\textstyle -\frac{3\l}4 \Jbr{z}^{-1} |B_1(z)| , & S_B(z) &{}\textstyle=\frac{\l}2 \Jbr{z}^{-1} |B_1(z)|
\end{align*}
if $A_1(z) =0$, by
\begin{align*}
	S_A(z) &{}=\textstyle -\frac{\l}2 \Jbr{z}^{-1} |A_1(z)| , & S_B(z) &{}\textstyle=\frac{3\l}4 \Jbr{z}^{-1} |A_1(z)|
\end{align*}
if $B_1(z) =0$, and by \eqref{E:SASB2D} if $A_1(z)\neq0$ and $B_1(z)\neq0$.
\end{remark}

\begin{remark}\label{R:genAssumption}
The generalized phase modification in Remark \ref{R:genSASB} allows us to have the same conclusion as in Theorem \ref{T:main} under another assumption.
One example is the following \emph{split assumption}: $A_1,B_1 \in H^{2.2}$ satisfy $\supp A_1 \cap \supp B_1 = \emptyset$.
In this case there is no interaction between two half-Klein Gordon waves in the leading order, which makes the situation considerably simple.
One can also consider a more general \emph{mixed condition}: 
$(A_1,B_1)$ is the sum of two pairs of functions of which support (as a $\C \times \C$-valued function) are disjoint; and one of the pair satisfies Assumption \ref{A:2D} and the other satisfies the split condition.
\end{remark}

\section{Reduction to the choice of asymptotic profile}\label{S:FVP}
To begin with, we recall a lemma from \cite{MS2}, which implies that the theorem holds if we choose a good asymptotic profile.

\begin{proposition}[\cite{MS2}] \label{prop:FVP2d}
Let $d=2$.
Let $\delta_0$ be a constant such that $1/2<\delta_0<1$. 
Then there exist constants $T>0$ and $\eps_0>0$ such that if 
a given profile $\widetilde{u}(t,x) \in C([T,\I);H^{1/2})$ satisfies
\begin{align}
\|\widetilde{u}(t)\|_{L^{\infty}_x} &{}\le \eps_0 t^{-1}, \label{eq:A1}\\
\|((\square +1 )\widetilde{u} - N(\widetilde{u}))(t)\|_{L^2_x}  &{} \le \eps_0 t^{-1-\delta_0}, \label{eq:A2}
\end{align}
for all $t\ge T$ 
then there exists a unique solution $u \in C([T,\infty);H^{1/2})$
for the equation \eqref{E:NLKG2} satisfying
\begin{equation}\label{E:pr2d}
\sup_{t \ge T} t^{\delta_0} (\|u-\widetilde{u}\|_{L^{\infty}((t,\infty);H^{1/2}_x)}
+\|u-\widetilde{u}\|_{L^4((t,\infty);L^{4}_x)}) <\infty.
\end{equation}
Moreover, if $\widetilde{u}(t)$ satisfies
\begin{equation}\label{E:condc2d}
	\liminf_{t\to\I} \norm{\Im \widetilde{u}(t)}_{L^2_x(\R^2)} >0
\end{equation}
in addition then $u(t)$ is a complex-valued solution.
\end{proposition}

The latter half is not in \cite{MS2} but the proof is easy.
We just note that \eqref{E:pr2d} and \eqref{E:condc2d} give us
$\norm{\Im u(t)}_{L^2_x}\gtrsim 1$ for large $t$.

One has a similar result for one dimensional case.

\begin{proposition} \label{prop:FVP1d}
Let $d=1$.
Let $\delta_0$ be a constant such that $1/4<\delta_0 \le5/12$. 
Then there exist constants $T>0$ and 
$\eps_0>0$ such that if 
a given profile $\widetilde{u}(t,x) \in C([T,\I);L^2) $ satisfies
\begin{align}
\|\widetilde{u}(t)\|_{L^{\infty}_x} &{}\le \eps_0 t^{-\frac12}, \label{eq:A3}\\
\|((\square +1 )\widetilde{u} - N(\widetilde{u}))(t)\|_{L^2_x}  &{} \le \eps_0 t^{-1-\delta_0}, \label{eq:A4}
\end{align}
for all $t\ge T$
then there exists a unique solution $u\in C([T,\infty);L^2_x)$ to (\ref{E:NLKG2}) satisfying
\[
\sup_{t \ge T} t^{\delta_0} 
(\|\langle\pt_{x}\rangle^{\frac54-3\delta_0}(u-\widetilde{u})\|_{L^{\infty}((t,\infty);L_x^{2})}
+\|u-\widetilde{u}\|_{L^{\frac{8}{3-4\delta_0}}((t,\infty);L^{\frac{4}{4\delta_0-1}}_x)}) <\infty.
\]
Moreover, if $\widetilde{u}(t)$ satisfies
\begin{equation}\label{E:condc1d}
	\liminf_{t\to\I} \norm{\Im \widetilde{u}(t)}_{L^2_x(\R)} >0
\end{equation}
in addition then $u(t)$ is a complex-valued solution.

\end{proposition}
This proposition follows by a standard contraction argument with Strichartz' estimate, as in the two dimensional case.

\begin{lemma}[Stirchartz estimate in one dimension]
Let $d=1$ and $T \in \R$.
Let $\Phi$ be an operator given by
\[
\Phi[g](t)
:=\int_t^{\infty} \sin((t-\tau)\langle\pt_{x}\rangle)\langle\pt_{x}\rangle^{-1}g(\tau)d\tau.
\]
Then, 
\begin{align*}
\|\Jbr{\d_{x}}^{\frac{3}{2r}}\Phi [g]
\|_{L_{t}^{q}([T,\infty),L_{x}^{r})}
&{}\lesssim \|\Jbr{\d_{x}}^{\frac12-\frac{3}{2\tilde{r}}}g\|_{L_{t}^{\tilde{q}'}([T,\infty),L_{x}^{\tilde{r}'})}
\end{align*}
holds for any $r, \tilde{r}\in[2,\infty]$ and $ q, \tilde{q}\in [4,\infty]$ obeying the scaling condition
\[\textstyle
\frac2q+\frac1r=\frac2{\tilde{q}}+\frac1{\tilde{r}}=\frac12.
\] 
\end{lemma}
The proof is also standard. We omit the detail.

\section{One dimensional case}\label{S:1d}

\subsection{Choice of asymptotic profile}

Before two dimensional case, we consider the one dimensional case in this section.
Compared with the two dimensional case, the one dimensional case is easier in several respects.
However, it contains the crucial point of the proof, which is the same as the two dimensional case.

The point is that an easy choice 
$\widetilde{u}=u_{\mathrm{ap}}$ \emph{does not satisfy} the assumption \eqref{eq:A4}. 
To see this, we expand $N(u_{\mathrm{ap}})$ and split it into the resonant and non-resonant parts:
\begin{align*}
	N(u_{\mathrm{ap}}) 
	={}&\l t^{-\frac32}{\bf 1}_{\{|x|<t\}}(t,x) 
	[(|A_1(\mu)|^2+2|B_1(\mu)|^2) A_1(\mu) e^{i\theta + iS_A(\mu)\log t} \\
	& \qquad\qquad \qquad \qquad+ (2|A_1(\mu)|^2+|B_1(\mu)|^2) B_1(\mu) e^{-i\theta + iS_B(\mu)\log t}]\\
	&+\l t^{-\frac32}{\bf 1}_{\{|x|<t\}}(t,x)[A_1(\mu)^2\overline{B_1(\mu)} e^{3i\theta+2iS_A(\mu)\log t-iS_B(\mu)\log t} \\
	& \qquad\qquad \qquad \qquad \quad + B_1(\mu)^2\overline{A_1(\mu)}e^{-3i\theta-iS_A(\mu)\log t+2iS_B(\mu)\log t}]\\
	=:{}& N_{\mathrm{r}}(u_{\mathrm{ap}}) + N_{\mathrm{nr}}(u_{\mathrm{ap}}).
\end{align*}
Notice that all terms in the right hand side have the same order of size in $L^2_x$ topology.
It will turn out that by the choice of the phase correction, $(\square +1 )u_{\mathrm{ap}}$ cancels only the resonant part
$N_{\mathrm{r}}(u_{\mathrm{ap}})$, that is, we have
\[
	\norm{(\square +1 )u_{\mathrm{ap}}-N_{\mathrm{r}}(u_{\mathrm{ap}})}_{L^2_x} \lesssim_{A_1,B_1} t^{-2}(\log t)^2
\]
for $t\ge 3$ (see Proposition \ref{P:uapNr1D}). This inequality also shows that
\[
	\norm{(\square +1 )u_{\mathrm{ap}}-N(u_{\mathrm{ap}})}_{L^2_x} \ge \norm{N_{\mathrm{nr}}(u_{\mathrm{ap}})}_{L^2_x} + o(t^{-1}).
\]	
The right hand side is $O(t^{-1})$ (at least if we choose $A_1$ and $B_1$ suitably).
In such a case, $u_{\mathrm{ap}}$ does not satisfy \eqref{eq:A4}.

In order to cancel out the non-resonant term $N_{\mathrm{nr}}(u_{\mathrm{ap}})$, we introduce a higher order correction $v_{\mathrm{ap}}$,
and define
\[
	\widetilde{u}(t,x)=u_{\mathrm{ap}}(t,x) + v_{\mathrm{ap}}(t,x),
\]
\begin{multline}\label{E:1dvp}
	v_{\mathrm{ap}}(t,x) = t^{-\frac32}{\bf 1}_{\{|x| < t\}}(t,x)
\Big( A_2(\mu)e^{3i\theta+2iS_{A}(\mu)\log t
-iS_{B}(\mu)\log t}\\
+B_2(\mu)e^{-3i\theta-iS_{A}(\mu)\log t +2iS_{B}(\mu)\log t}\Big),
\end{multline}
where the phase functions $S_{A}$ and $S_{B}$ are given by \eqref{E:SASB1D} and 
\begin{align}
A_{2}(\xi)=&-\tfrac{\lambda}{8}A_{1}^{2}(\xi)\overline{B_{1}(\xi)},&
B_{2}(\xi)=&-\tfrac{\lambda}{8}\overline{A_{1}(\xi)}B_{1}^{2}(\xi).\label{pn}
\end{align}

Next proposition ensures that $u_{\mathrm{ap}}+v_{\mathrm{ap}}$ is one of the profiles
which satisfy the assumptions \eqref{eq:A1} and \eqref{eq:A2}.

\begin{proposition} \label{P:ua} 
Assume $A_{1},B_{1}\in H^{2,5/2}$. 
Let $u_{\mathrm{ap}}$ and $v_{\mathrm{ap}}$ be defined by \eqref{E:uap1D} and \eqref{E:1dvp}, respectively. 
Let $\widetilde{u}=u_{\mathrm{ap}}+v_{\mathrm{ap}}$.
Then,  
\begin{align}
\|\widetilde{u}(t)\|_{L^{\infty}_x} &{}\lesssim t^{-\frac12}
(\|A_1\|_{L^\infty}+\|B_1\|_{L^\infty})\Jbr{\|A_1\|_{L^{\infty}}+\|B_1\|_{L^\infty}}^2,
\label{eq:ua1}
\end{align}
\begin{multline}
\|(\Box+1)\widetilde{u}(t) -N(\widetilde{u}(t))\|_{L^2_x} \\
	\lesssim t^{-2}(\log t)^{2}(\|A_1\|_{H^{2,\frac52}}+\|B_1\|_{H^{2,\frac52}})\Jbr{\|A_1\|_{H^{2,\frac52}}+\|B_1\|_{H^{2,\frac52}}}^8
\label{eq:ua2}
\end{multline}
hold for $t\ge 3$.
In particular, $\widetilde{u}(t)$ satisfies \eqref{eq:A3} and \eqref{eq:A4} for any $\delta_0<1$
if $A_1,B_1 \in H^{2,5/2}$ are small in $L^\I(\R^2)$ and if $T$ is large.
Further, if the set $\{ \xi \in \R \ |\ |A_1(\xi)|\neq |B_1(\xi)| \}$ has positive measure then $\widetilde{u}(t)$ satisfies \eqref{E:condc1d}.
\end{proposition}

\subsection{Hyperbolic coordinate}

To prove Proposition \ref{P:ua}, as in the previous works \cite{DFX,S1},
we introduce the hyperbolic coordinate
\begin{equation}\label{E:hc}
	t = \tau \cosh \sigma,
	\quad x = \tau \sinh \sigma
\end{equation}
for $|x|<t$ and $t\ge 1$.
More explicitly, they are given by
\begin{equation}\label{E:ihc}
	\tau = \sqrt{t^2 - x^2}, \quad \sigma = \tanh^{-1} \( \frac{x}{t} \),
\end{equation}
and the range is $\{ (\tau,\sigma) \in \R^2 \ |\  \sigma \in \R,\,\tau \ge (\cosh \sigma)^{-1}\}$.
In the hyperbolic coordinate, $\theta(t,x)=-\tau$ and $\mu(t,x) = \sinh \sigma$, and so
$u_{\mathrm{ap}}$ and $v_{\mathrm{ap}}$ given by \eqref{E:uap1D} and \eqref{E:1dvp}, respectively, 
can be rewritten as
\begin{align*}
	u_{\mathrm{ap}}(t,x) ={}& \tau^{-\frac12}(\cosh \sigma)^{-\frac12}( A_1(\sinh \sigma) e^{-i\tau + iS_A(\sinh \sigma)(\log \tau + \log \Jbr{\sinh \sigma})}\\
	&+ B_1(\sinh \sigma) e^{i\tau + iS_B(\sinh \sigma)(\log \tau + \log \Jbr{\sinh \sigma})})
\end{align*}
and
\begin{align*}
	&v_{\mathrm{ap}}(t,x) \\
	&{}= \tau^{-\frac32}(\cosh \sigma)^{-\frac32}(A_2(\sinh \sigma) e^{-3i\tau + i(2S_A(\sinh \sigma)-S_B(\sinh \sigma))(\log \tau + \log \Jbr{\sinh \sigma})}\\
	&\qquad + B_2(\sinh \sigma) e^{3i\tau + i(-S_A(\sinh \sigma)+2S_B(\sinh \sigma))(\log \tau + \log \Jbr{\sinh \sigma})}).
\end{align*}
For each fixed $t\ge 1$,
\[\textstyle
	\norm{v(\sinh \sigma, \tau)}_{L^2_x(|x|<t)}
	= t^\frac12 \norm{\Jbr{z}^{-\frac32} v\(z, \frac{t}{\Jbr{z}} \)}_{L^2_z(\R)}.
\]
It is well-known that
\[
	\square_{t,x} = \partial_\tau^2 + \tau^{-1} \partial_\tau -\tau^{-2} \partial_\sigma^2.
\]
By a further calculation, we have the following.
\begin{lemma}\label{L:hc1d}
Let $n, m  \in \R$.
If a function $v(t,x)$ is of the form
\[
	v(t,x) = \tau^{-\frac12- m } (\cosh \sigma)^{-\frac12- m } e^{in \tau } H(\sinh \sigma ,\tau)
\]
in the hyperbolic coordinate with some function $H(\cdot,\cdot)$. Then, we have
\begin{align*}
	[(\square_{t,x} + 1) v](t,x) ={}& t^{-\frac32} e^{in\tau}f_1(\sinh \sigma,\tau) + t^{-\frac32} e^{in\tau}f_2(\sinh \sigma,\tau)\\
	&{}+ t^{-\frac32- m  } e^{in\tau}R_1(\sinh \sigma,\tau) + t^{-\frac52- m }e^{in\tau} R_2(\sinh \sigma,\tau),
\end{align*}
where
\begin{align*}
	f_1(z,\tau) &{}= (1-n^2)\tau^{1- m }\Jbr{z}^{1- m } H(z,\tau),\\
	f_2(z,\tau) &{}=\textstyle 2in \tau^{1- m } \Jbr{z}^{1- m }\frac{\partial}{\partial \tau} H(z,\tau),\\
	R_1(z,\tau) &{}=\textstyle  m  \Jbr{z}\(  - {2in} - 2\frac{\partial}{\partial \tau} +\frac{ m +1}{\tau}  \)
	 H(z,\tau),\\
	R_2(z,\tau) &{}\textstyle = \(\Jbr{z}^2 \tau^2 \frac{\partial^2}{\partial \tau^2}  -
	\Jbr{z}^{4+ m } \frac{\partial^2}{\partial z^2}\Jbr{z}^{- m } + \frac3{4} \)
	 H(z,\tau).
\end{align*}
\end{lemma}
\subsection{Cancellation of the resonant part}

\begin{proposition}\label{P:uapNr1D} 
Let $N_{\mathrm{r}}$ be the resonant part of the nonlinearity, that is,
\begin{align*}
N_{\mathrm{r}}(u_{\mathrm{ap}})=t^{-\frac32}{\bf 1}_{\{|x|<t\}}(t,x)
(&(- 2\Jbr{\mu}S_A )A_{1}(\mu)e^{i\theta+iS_A(\mu)\log t}
\\
&
+(2\Jbr{\mu}S_B)B_{1}(\mu)e^{-i\theta+iS_B(\mu)\log t}).
\end{align*}
Then, we have
\begin{multline}\label{N2}
\|(\square+1)u_{\mathrm{ap}}-N_{\mathrm{r}}(u_{\mathrm{ap}})\|_{L_{x}^{2}}\\
\lesssim t^{-2}(\log t)^{2}(\|A_1\|_{H^{2,\frac52}}+\|B_1\|_{H^{2,\frac52}})
\Jbr{\|A_1\|_{H^{2,\frac52}}+\|B_1\|_{H^{2,\frac52}}}^{4},
\end{multline}
\end{proposition}
\begin{proof}
By Lemma \ref{L:hc1d} with $n=\mp1$ and $ m =0$,
\[
	(\square+1)u_{\mathrm{ap}}-N_{\mathrm{r}}(u_{\mathrm{ap}})
	=  t^{-\frac52}(e^{-i\tau} R_A(\sinh \sigma,\tau) + e^{i\tau} R_B(\sinh \sigma,\tau)),
\]
where
\[\textstyle
	R_A(z,\tau) = (\Jbr{z}^2 \tau^2 \frac{\partial^2}{\partial \tau^2}  -
	\Jbr{z}^{4} \frac{\partial^2}{\partial z^2} + \frac3{4} )
	(A_1(z) e^{iS_A(z)(\log \tau + \log \Jbr{z})})
\]
and
\[\textstyle
	R_B(z,\tau) = (\Jbr{z}^2 \tau^2 \frac{\partial^2}{\partial \tau^2}  -
	\Jbr{z}^{4} \frac{\partial^2}{\partial z^2} + \frac3{4} )
	(B_1(z) e^{iS_B(z)(\log \tau+\log \Jbr{z})}).
\]
Recall that the support of $u_{\mathrm{ap}}$ is a subset of $\{|x| \le t\}$. Hence,
\begin{align*}
	t^2\norm{(\square+1)u_{\mathrm{ap}}-N_{\mathrm{r}}(u_{\mathrm{ap}})}_{L^2_x}
	\le {}& \textstyle \norm{\Jbr{z}^{-\frac32} R_A\(z , \frac{t}{\Jbr{z}}\)}_{L^2_z} \\
	&\textstyle + \norm{\Jbr{z}^{-\frac32} R_B\(z , \frac{t}{\Jbr{z}}\)}_{L^2_z},
\end{align*}
from which the desired estimate follows.
\end{proof}
\subsection{Cancellation of the non-resonant part}

\begin{proposition}\label{P:vn1d} 
Let $N_{\mathrm{nr}}$ be the non-resonant part of the nonlinearity, that is,
\begin{align*}
N_{\mathrm{nr}}(u_{\mathrm{ap}})=t^{-\frac32}{\bf 1}_{\{|x|<t\}}(t,x)
&(-8A_2(\mu)e^{3i\theta+2iS_{A}(\mu)\log t
-iS_{B}(\mu)\log t}\\
&
-8B_2(\mu)e^{-3i\theta-iS_{A}(\mu)\log t
+2iS_{B}(\mu)\log t}),
\end{align*}
Let $v_{\mathrm{ap}}$ be as in \eqref{E:1dvp}.
Then, we have 
\begin{multline}\label{E:N3}
\|(\Box+1)v_{\mathrm{ap}}-N_{\mathrm{nr}}(u_{\mathrm{ap}})\|_{L_{x}^{2}}\\
\lesssim t^{-2}(\norm{A_1}_{H^{2,\frac52}}+\norm{B_1}_{H^{2,\frac52}})^{3}\Jbr{\norm{A_1}_{H^{2,\frac52}}+\norm{B_1}_{H^{2,\frac52}}}^{4}.
\end{multline}
\end{proposition}
\begin{proof}
We decompose $v_{\mathrm{ap}}=v_{-3}+ v_{3}$ with
\begin{align*}
	v_{-3}(t,x) ={}& \tau^{-\frac32}(\cosh \sigma)^{-\frac32} e^{-3i\tau} [A_2(z) e^{i(2S_A(z)-S_B(z))(\log \tau + \log \Jbr{z})}]_{|z=\sinh \sigma},\\
	v_{3}(t,x) ={}& \tau^{-\frac32}(\cosh \sigma)^{-\frac32} e^{3i\tau} [B_2(z) e^{i(-S_A(z)+2S_B(z))(\log \tau + \log \Jbr{z})}]_{|z=\sinh \sigma}.
\end{align*}
We first estimate $v_{-3}$. Apply Lemma \ref{L:hc1d} with $n=-3$ and $m=1$ to obtain
\begin{align*}
	[(\square + 1) v_{-3}](t,x) ={}& t^{-\frac32} e^{-3i\tau}f_1(\sinh \sigma,\tau) + t^{-\frac32} e^{-3i\tau}f_2(\sinh \sigma,\tau)\\
	&{}+ t^{-\frac52  } e^{-3i\tau}R_1(\sinh \sigma,\tau) + t^{-\frac72 }e^{-3i\tau} R_2(\sinh \sigma,\tau),
\end{align*}
where
\begin{align*}
	f_1(z,\tau) &{}= -8 A_2(z) e^{i(2S_A(z)-S_B(z))(\log \tau + \log \Jbr{z})},\\
	f_2(z,\tau) &{}= -6it^{-1} \Jbr{z}(2S_A-S_B)A_2(z) e^{i(2S_A(z)-S_B(z))(\log \tau + \log \Jbr{z})},\\
	R_1(z,\tau) &{}=  6i\Jbr{z}A_2(z) e^{i(2S_A(z)-S_B(z))(\log \tau + \log \Jbr{z})}\\
	 &\quad  +2 t^{-1} 
	\Jbr{z}^2 A_2(z) e^{i(2S_A(z)-S_B(z))(\log \tau + \log \Jbr{z})}\\
	 &\quad -2t^{-1} \Jbr{z}^2(2S_A-S_B)A_2(z) e^{i(2S_A(z)-S_B(z))(\log \tau + \log \Jbr{z})},\\
	R_2(z,\tau) &{}= \textstyle \(\Jbr{z}^2 \tau^2 \frac{\partial^2}{\partial \tau^2}  -
	\Jbr{z}^{5 } \frac{\partial^2}{\partial z^2}\Jbr{z}^{- 1 } + \frac3{4} \)A_2(z) e^{i(2S_A(z)-S_B(z))(\log \tau + \log \Jbr{z})}.
\end{align*}
Note that $f_1$ cancels the first half of $N_{\mathrm{nr}}(u_{\mathrm{ap}})$.
The other terms are the remainder.
We have
\[
	\norm{f_2(\sinh \sigma,\tau)}_{L^2_x}
	\lesssim t^{-\frac12}(\norm{A_1}_{L^{10}}+\norm{B_1}_{L^{10}})^5,
\]
\begin{multline*}
	\norm{R_1(\sinh \sigma,\tau)}_{L^2_x} \lesssim t^{\frac12}(\norm{A_1}_{L^6}+\norm{B_1}_{L^6})^3 \\
	+ t^{-\frac12}(\norm{A_1}_{H^{1,\frac12}}+\norm{B_1}_{H^{1,\frac12}})^3\Jbr{\norm{A_1}_{L^\I}+\norm{B_1}_{L^\I}}^2,
\end{multline*}
and
\begin{multline*}
	\norm{R_2(\sinh \sigma,\tau)}_{L^2_x} \\
	\lesssim t^{\frac12}(\log t)^2(\norm{A_1}_{H^{2,\frac52}}+\norm{B_1}_{H^{2,\frac52}})^3\Jbr{\norm{A_1}_{H^{2,\frac52}}+\norm{B_1}_{H^{2,\frac52}}}^4.
\end{multline*}
Similar estimates hold for $v_3$. Hence, we obtain the desired estimate.
\end{proof}
\subsection{Proof of Proposition \ref{P:ua}}

Let us now complete the proof of the proposition.
\begin{proof}[Proof of Proposition \ref{P:ua}]
The first estimate \eqref{eq:ua1} is immediate by definitions of $u_{\mathrm{ap}}$ and $v_{\mathrm{ap}}$.
Let us prove \eqref{eq:ua2}. By the triangle inequality,
\begin{multline*}
	\norm{(\square +1)\widetilde{u} - N(\widetilde{u})}_{L^2_x}
	\le \norm{(\square +1)u_{\mathrm{ap}} - N_{\mathrm{r}}(u_{\mathrm{ap}})}_{L^2_x}
	\\+ \norm{(\square +1)v_{\mathrm{ap}} - N_{\mathrm{nr}}(u_{\mathrm{ap}})}_{L^2_x}
	+\norm{N(u_{\mathrm{ap}})-N(u_{\mathrm{ap}}+v_{\mathrm{ap}})}_{L^2_x}.
\end{multline*}
The first two terms of the right hand side are handled by Propositions \ref{P:uapNr1D} and \ref{P:vn1d}.
Finally, 
\begin{align*}
	\norm{N(u_{\mathrm{ap}})(t)-N(u_{\mathrm{ap}}+v_{\mathrm{ap}})(t)}_{L^2_x}
	&{}\lesssim (\norm{u_{\mathrm{ap}}(t)}_{L^\I}+ \norm{v_{\mathrm{ap}}(t)}_{L^\I})^2\norm{v_{\mathrm{ap}}(t)}_{L^2}\\
	&{}\lesssim t^{-2} \norm{A_1}_{L^2}\Jbr{\norm{A_1}_{L^\I} + \norm{B_1}_{L^\I}}^8.
\end{align*}
Combining the above estimates, we complete the proof of \eqref{eq:ua2}.

Suppose that $ \{\xi \in \R \ |\ |A_1(\xi)|\neq |B_1(\xi)|\}$ has positive measure.
One sees from \eqref{E:SASB1D} that the set is equal to $ \{\xi \in \R \ |\ S_A(\xi)+S_B(\xi) \neq 0\}$.
Since $A_1$ and $B_1$ belong to $C^1(\R) \cap L^2(\R)$, there exists a compact set $\Omega\subset \{\xi \in \R \ |\ |A_1(\xi)|\neq |B_1(\xi)|\}$
such that $\inf_{\xi \in \Omega}|\partial_\xi ( S_A(\xi) + S_B(\xi))| > 0$ and $|\Omega|>0$. 
Indeed, if such $\Omega$ does not exist then $S_A(\xi)+S_B(\xi)$ is a nonzero constant.
Hence, $|A_1(z)|^2-|B_1(z)|^2=C\Jbr{z}$. However, this contradicts with $A_1, B_1 \in L^2(\R)$.

For this $\Omega$, we have
\begin{align*}
	&\norm{\Im u_{\mathrm{ap}}(t)}_{L^2_x}^2 \\
	&{}\textstyle = \norm{\Im(A_1(z) e^{-i\Jbr{z}^{-1}t + iS_A(z)\log t}+ B_1(z) e^{i\Jbr{z}^{-1}t + iS_B(z)\log t})}_{L^2_z(\R)}^2 \\
	&{}\textstyle \ge \norm{\Im(A_1(z) e^{-i\Jbr{z}^{-1}t + iS_A(z)\log t}+ B_1(z) e^{i\Jbr{z}^{-1}t + iS_B(z)\log t})}_{L^2_z({\Omega})}^2\\
	&{}= \tfrac12 \norm{ A_1 }_{L^2(\Omega)}^2 + \tfrac12 \norm{ B_1 }_{L^2(\Omega)}^2 + o(1)
\end{align*}
as $t\to \I$,
where the small order term consists of cross terms such as
\[
	\int_\Omega A_1 (z) B_1(z) e^{i(S_A(z)+S_B(z))\log t}dz.
\] 
This term is small by using $\inf_{\xi \in \Omega}|\partial_\xi ( S_A(\xi) + S_B(\xi))| > 0$.
The all other cross terms are also small by a similar reason.
Moreover, $\norm{v_{\mathrm{ap}}(t)}_{L^2_x} = O(t^{-1})$
as $t\to\I$. Thus, \eqref{E:condc1d} holds. We complete the proof.
\end{proof}

\section{Two dimensional case}\label{S:2d}

\subsection{Choice of asymptotic profile}
Let us move to the two dimensional case.
By the same reason as in the one dimensional case, the easy choice $\widetilde{u}=u_{\mathrm{ap}}$
does not satisfy \eqref{eq:A2}.
We put a higher order correction, $v_{\mathrm{ap}}$, to cancel out the non-resonant part of the nonlinearity.
In the two dimensional case, the non-resonant part takes a complicated form.
The higher correction terms are defined
by utilizing the Fourier series expansion of the nonlinearity which we observed in Section \ref{S:observation}.

For second asymptotic profile, we let
\begin{align}
A_n(\mu) &=  \tfrac{\l}{1-(2n-1)^2} L_{n-1}(\zeta(\mu))|A_1(\mu)|A_1(\mu)e^{i(n-1)\alpha(\mu)}, \\
B_n(\mu) &=  \tfrac{\l}{1-(2n-1)^2} L_{n-1}(1/\zeta(\mu))|B_1(\mu)|B_1(\mu)e^{-i(n-1)\alpha(\mu)},
\end{align}
for $n\ge2$ and define 
\begin{equation}\label{E:vap2d}
\begin{aligned}
v_{\mathrm{ap}}(t,x) 
&{}= t^{-2}{\bf 1}_{\{|x| < t\}}(t,x)
\sum_{n=2}^\infty\bigg\{
A_n(\mu)e^{i(2n-1)\theta + inS_{A}(\mu)\log t - i(n-1)S_{B}(\mu)\log t}\\
&\qquad \qquad \qquad \qquad  + B_n(\mu)e^{-i(2n-1)\theta - i(n-1)S_{A}(\mu)\log t + inS_{B}(\mu)\log t}
\bigg\}.
\end{aligned}
\end{equation}

With this choice, one has the following.

\begin{proposition} \label{P:ua2D} 
Assume that $\phi_{0},\phi_{1}$ satisfy Assumption \ref{A:2D}. 
Let $v_{\mathrm{ap}}(t,x)$ be defined by \eqref{E:vap2d} and $\widetilde{u}:=u_{\mathrm{ap}}+v_{\mathrm{ap}}$.
Then,  
\begin{equation}
\|\widetilde{u}(t)\|_{L^{\infty}_x} \lesssim_{\rho_0} t^{-1}
(\norm{A_1}_{L^\I}+\norm{B_1}_{L^\I})\Jbr{\norm{A_1}_{L^\I}+\norm{B_1}_{L^\I}}^2,
\label{eq:ua12d}
\end{equation}
\begin{multline}
\|(\Box+1)\widetilde{u}(t) -N(\widetilde{u}(t))\|_{L^2_x} \\
\lesssim_{\rho_0} t^{-2}(\log t)^2(\norm{A_1}_{H^{2,2}}+\norm{B_1}_{H^{2,2}})\Jbr{\norm{A_1}_{H^{2,2}}+\norm{B_1}_{H^{2,2}}}^3
\label{eq:ua22d}
\end{multline}
hold for $t\ge 3$, where $\rho_0$ is the constant given in Assumption \ref{A:2D}.
In particular, $\widetilde{u}(t)$ satisfies \eqref{eq:A3} and \eqref{eq:A4} for any $\delta_0<1$ if $A_1,B_1\in H^{2,2}$ are
small in $L^\I(\R^2)$ compared with $\rho_0$ and if $T$ is large.
Furthermore, if the set $\{ \xi \in \R^2 \ |\ |A_1(\xi)|\neq |B_1(\xi)| \}$ has positive measure then $\widetilde{u}(t)$ satisfies \eqref{E:condc2d}.
\end{proposition}

The key ingredient is the decomposition of $N(u_{\mathrm{ap}})$.
One deduces from \eqref{E:Fourier2D} and \eqref{E:reflection2D} that
\begin{equation}\label{E:resnonres2d}
\begin{aligned}
	&N(u_{\mathrm{ap}}) \\
	&=\l t^{-2}{\bf 1}_{\{ |x|<t \} } \bigg[ |A_1|A_1 L_0(\zeta)  e^{i\theta} e^{iS_A\log t}
	+ |B_1|B_1 L_0(1/\zeta)  e^{-i\theta} e^{iS_B\log t}\\
	&\quad{}+ \sum_{n \ge 1} |A_1|A_1 L_n(\zeta) e^{iS_A\log t}e^{in(\alpha+(S_A-S_B)\log t)} e^{i(2n+1)\theta}\\
	&\quad{}+ \sum_{n\le -2} |B_1|B_1 L_{-n-1}(1/\zeta) e^{iS_B\log t} e^{i(n+1)(\alpha+(S_A-S_B)\log t)} e^{i(2n+1)\theta}\bigg]\\
	&=t^{-2}{\bf 1}_{\{ |x|<t \} } \big[ -2 \Jbr{\mu} S_A A_1  e^{i\theta} e^{iS_A\log t}
	+ 2 \Jbr{\mu} S_B B_1  e^{-i\theta} e^{iS_B\log t}\big]\\
	&\quad{}+t^{-2}{\bf 1}_{\{ |x|<t \} } \sum_{n\ge 2}\bigg[ (1-(2n-1)^2)A_n e^{iS_A\log t} e^{i(n-1)(S_A-S_B)\log t} e^{i(2n-1)\theta}\\	
	&\qquad\qquad\qquad\quad{}+(1-(2n-1)^2)B_n e^{iS_B\log t}e^{i(n-1)(S_B-S_A)\log t} e^{-i(2n-1)\theta}\bigg]\\
	&=: N_{\mathrm{r}}(u_{\mathrm{ap}}) + N_{\mathrm{nr}}(u_{\mathrm{ap}}).
\end{aligned}
\end{equation}

\subsection{Hyperbolic coordinate}
To prove Proposition \ref{P:ua2D}, we introduce the hyperbolic coordinate $(\tau, \sigma ,\omega) \in \R \times \R_{\ge0} \times \mathbb{T}$
by
\[
	t = \tau \cosh \sigma, \quad x_1 = \tau \sinh \sigma \cos \omega,\quad
	x_2 = \tau \sinh \sigma \sin \omega.
\]
for $|x|<t$ and $t\ge1$. We have explicit definition:
\[\textstyle
	\tau = \sqrt{t^2 -x_1^2 - x_2^2},\quad \sigma = \tanh^{-1}\frac{\sqrt{x_1^2+x_2^2}}{t},\quad
	\omega = \arg (x_1+ix_2).
\]
The range of the new variables is $\{ (\tau, \sigma ,\omega) \in \R \times \R_{\ge0} \times \mathbb{T}\ |\ \tau \ge (\cosh \sigma)^{-1} \}$.
Remark that $\theta (t,x)=-\tau$ and $\mu=(\mu_1,\mu_2)=(\sinh\sigma \cos \omega , \sinh \sigma \sin \omega)$.
Moreover, 
\begin{equation}\label{E:intchange2D} \textstyle
	\norm{v(\sinh \sigma \cos \omega, \sinh \sigma \sin \omega , \tau)}_{L^2_x(|x|<t)}
	= t \norm{\Jbr{z}^{-2} v\(z_1,z_2, \frac{t}{\Jbr{z}} \)}_{L^2_z(\R^2)}
\end{equation}
for each fixed $t\ge 1$.

One sees that
\[\textstyle
	\square_{t,x} + 1 = \partial_\tau^2 +1
	+\frac2\tau  \partial_\tau - \frac1{\tau^2}  \partial_\sigma^2
	  - \frac{1}{\tau^2 \tanh \sigma} \partial_\sigma
	 - \frac{1}{\tau^2 (\sinh \sigma)^2}\partial_\omega^2.
\]
A computation shows the following.
\begin{lemma}\label{L:hc2D}
Let $n, m  \in \R$.
If a function $v(t,x_1,x_2)$ is of the form
\[
	v(t,x_1,x_2) = \tau^{-1- m } (\cosh \sigma)^{-1- m } e^{in \tau } H(\sinh \sigma \cos \omega, \sinh \sigma \sin \omega ,\tau)
\]
in the hyperbolic coordinate with some function $H(\cdot,\cdot,\cdot)$. Then, we have
\begin{align*}
	(\square_{t,x} +1)v ={}& t^{-2} e^{in\tau} \sum_{j=1,2} f_j(\sinh \sigma \cos \omega , \sinh \sigma \sin \omega, \tau)\\
	&{}+t^{-2-m} e^{in\tau} R_1(\sinh \sigma \cos \omega , \sinh \sigma \sin \omega, \tau)\\
	&{}+ t^{-3-m} e^{in\tau} R_2(\sinh \sigma \cos \omega , \sinh \sigma \sin \omega, \tau)
\end{align*}
where
\begin{align*}
	f_1(z_1,z_2,\tau)&{}= (1-n^2) \tau^{1-m} \Jbr{z}^{1-m} H(z_1,z_2, \tau),\\
	f_2(z_1,z_2,\tau)&{}=\textstyle 2in \tau^{1-m} \Jbr{z}^{1-m} \frac{\partial}{\partial \tau} H(z_1,z_2,\tau),\\
	R_1(z_1,z_2,\tau)&{}\textstyle= m\Jbr{z} \( -2in - 2\frac{\partial}{\partial \tau} + \frac{m+1}{\tau} \) H(z_1,z_2,\tau),\\
	R_2(z_1,z_2,\tau)&{}\textstyle= \Jbr{z}^2 \tau^2 \frac{\partial^2}{\partial \tau^2} H(z_1,z_2,\tau)
	 - \Jbr{z}^{2+m}\Delta_z (\Jbr{z}^{-m}H) (z_1,z_2,\tau)\\
	&\textstyle\quad - \Jbr{z}^{2+m} \begin{pmatrix}z_1 \\ z_2 \end{pmatrix}^t \nabla_z^2 (\Jbr{z}^{-m}H(z_1,z_2,\tau))\begin{pmatrix}z_1 \\ z_2 \end{pmatrix} +2 H (z_1,z_2,\tau).
\end{align*}
Here, $\Delta_z=\partial_{z_1}^2 + \partial_{z_2}^2$ and $\nabla^2_z A$ denotes the Hessian matrix $(\partial_{z_i}\partial_{z_j}A)_{i,j}$.
\end{lemma}

\subsection{Cancellation of the resonant part}

\begin{proposition}\label{P:uapNr2D} 
Let $N_{\mathrm{r}}(u_{\mathrm{ap}})$ be the resonant part given in \eqref{E:resnonres2d}.
Then, we have
\begin{multline*}
\|(\square+1)u_{\mathrm{ap}}-N_{\mathrm{r}}(u_{\mathrm{ap}})\|_{L_{x}^{2}}\\
\lesssim t^{-2}(\log t)^{2}(\|A_{1}\|_{H^{2,2}}+\|B_{1}\|_{H^{2,2}})
\Jbr{\|A_{1}\|_{H^{2,2}}+\|B_{1}\|_{H^{2,2}}}^{4}
\end{multline*}
for $t\ge 3$.
\end{proposition}

Before the proof of the proposition, we prepare one lemma.
\begin{lemma}\label{L:Sest2D}
Let $A_1,B_1\in H^{2,2}$ satisfy Assumption \ref{A:2D}. 
Then, 
\begin{align*}
	\norm{\Jbr{z}^{3}\nabla S_A(z)}_{L^4_z}
	\lesssim_{\rho_0} |\l|\norm{A_1}_{H^{2,2}}(1+\norm{\nabla \zeta}_{L^\I+L^4})
\end{align*}
and
\begin{multline*}
	\norm{\Jbr{z}^{5}A_1(z) \partial_{z_j}\partial_{z_k} S_A(z)}_{L^2_z}\\
	\lesssim_{\rho_0} |\l|\norm{A_1}_{H^{2,2}}^2(1+\norm{\nabla \zeta}_{L^\I+L^4}^2+\norm{\partial_{z_j}\partial_{z_k} \zeta }_{L^\I+L^2}),
\end{multline*}
where the implicit constants depend on the constant $\rho_0>1$ given in Assumption \ref{A:2D} and  
\[
	\max_{k=0,1,2} \sup_{\zeta \in [\rho_0^{-1},\rho_0]} L_0^{(k)}(\zeta).
\]
Similar estimates hold for $\nabla S_B$ and $B_1 \partial_{z_j}\partial_{z_k} S_B$.
\end{lemma}

\begin{proof}
Just for notational simplicity, we let $\l=2$.
The first estimate is immediate from
\begin{equation*}
\begin{aligned}
\partial_{ z_j} S_{A} (z)
&= -\< z\>^{-3} z_j |A_1|
L_0\left(\zeta(z)\right)\\
&+ 2^{-1}\< z\>^{-1}
\left(A_1^{-{1}/{2}}\overline{A_1}^{{1}/{2}}\partial_{ z_j} A_1
+ A_1^{{1}/{2}}\overline{A_1}^{-{1}/{2}}\partial_{ z_j} \overline{A_1}
\right)L_0\left(\zeta(z)\right)\\
&+ \< z\>^{-1}|A_1|L_0'\left( \zeta(z) \right)
 \partial_{ z_j} \zeta(z)
\end{aligned}
\end{equation*}
and H\"older's inequality.
Let us move to the second estimate. 
For the second order derivative for $S_{A}$, we have 
\begin{equation*}
\begin{aligned}
& \partial_{ z_j}\partial_{ z_k} S_{A}(z)\\
&{}= -\< z\>^{-5} z_k z_j|A_1(z)|
L_0(\zeta(z))
-\delta_{jk}\< z\>^{-3}|A_1(z)|L_0(\zeta(z))\\
&\quad -\< z\>^{-3} z_k \partial_{ z_j}|A_1(z)|
L_0(\zeta(z))
- \< z\>^{-3} z_j \partial_{ z_j}|A_1(z)|L_0(\zeta(z))\\
&\quad - \< z\>^{-3} z_k|A_1(z)|L'_0\left(\zeta(z)\right)\partial_{ z_j}\zeta(z)
- \< z\>^{-3} z_j|A_1(z)| L_0'\left(\zeta(z)\right) \partial_{ z_k}\zeta(z) \\
&\quad 
+ \< z\>^{-1}\partial_{ z_j}\partial_{ z_k}|A_1(z)|
L_0\left(\zeta(z)\right)
+ \< z\>^{-1}\partial_{ z_k}|A_1(z)|
L_0'\left(\zeta(z)\right)\partial_{ z_j}\zeta(z)\\
&\quad + \< z\>^{-1}\partial_{ z_j}|A_1(z)|
L_0'(\zeta(z)) \partial_{ z_k}\zeta(z)
+ \< z\>^{-1}|A_1(z)|
L_0''(\zeta(z)) \partial_{ z_k}\zeta(z) \partial_{ z_j}\zeta(z)\\
&\quad 
+ \< z\>^{-1}|A_1(z)|
L_0'\left(\zeta(z)\right)
\partial_{ z_j}\partial_{ z_k}\zeta(z).
\end{aligned}
\end{equation*}
The desired estimate follows from H\"older's estimate, Sobolev embedding, and
the estimates
\[
	|\partial_{z_j}|A_1|| \lesssim |\partial_{z_j}A_1|,\quad
	|A_1\partial_{z_j}\partial_{z_k}|A_1|| \lesssim |A_1| |\partial_{z_j}\partial_{z_k}A_1| +
	|\partial_{z_j}A_1||\partial_{z_k}A_1|
\]
for almost all $z\in\R^2$.
\end{proof}

\begin{proof}[Proof of Proposition \ref{P:uapNr2D}]
By Lemma \ref{L:hc2D} with $n=\mp1$ and $ m =0$,
\[
	(\square+1)u_{\mathrm{ap}}-N_{\mathrm{r}}(u_{\mathrm{ap}})
	=  t^{-3}  \sum_{n=\mp1}e^{in\tau} R_{2,n}(\sinh \sigma \cos \omega, \sinh \sigma \sin \omega ,\tau) 
\]
for $(t,x)$ such that $t>1$ and $|x|<t$, where
\begin{align*}
	R_{2,n}(z_1,z_2,\tau) ={}&\textstyle \Jbr{z}^2 \tau^2 \frac{\partial^2}{\partial \tau^2} H_n(z_1,z_2,\tau)
	 - \Jbr{z}^{2}\Delta_z H_n (z_1,z_2,\tau)\\
	&\quad - \Jbr{z}^{2} \begin{pmatrix}z_1 \\ z_2 \end{pmatrix}^t \nabla_z^2 H_n(z_1,z_2,\tau)\begin{pmatrix}z_1 \\ z_2 \end{pmatrix} +2 H_n (z_1,z_2,\tau)
\end{align*}
and
\begin{align*}
	H_{-1}(z_1,z_2,\tau) &{}= 
	A_1(z) e^{iS_A(z)(\log \tau+\log \Jbr{z})},\\
	H_1(z_1,z_2,\tau) &{}= 
	B_1(z) e^{iS_B(z)(\log \tau+\log \Jbr{z})}.
\end{align*}

We consider the case $n=-1$. For fixed $t\ge1$, we see from \eqref{E:intchange2D} that
\begin{align*}\textstyle
	\norm{\Jbr{z}^2 \tau^2 \frac{\partial^2}{\partial \tau^2} H_{-1}(z_1,z_2,\tau)}_{L^2_x( |x|<t )}
	&{}\lesssim t \norm{|A_1(z)|(|S_A(z)|+|S_A(z)|^2)}_{L^2_z}\\
	&{}\lesssim_{\rho_0} t \norm{A_1}_{H^{2,2}}^2\Jbr{\norm{A_1}_{H^{2,2}}}
\end{align*}
and, by means of Lemma \ref{L:Sest2D},
\begin{align*}
	&\norm{\Jbr{z}^{4}\abs{\nabla_{z}^2 H_{-1} (z_1,z_2,\tau)}}_{L^2_x( |x|<t )}\\
	&\lesssim t\norm{\Jbr{z}^{2}|\nabla_{z}^2A_1(z)|}_{L^2_z}
	+t\norm{\Jbr{z}^{2}|\nabla A_1(z)|(|\nabla S_A|\log t+\Jbr{z}^{-1}| S_A|)}_{L^2_z}\\
	& \quad + t\norm{\Jbr{z}^{2}|A_1(z)||\nabla S_A|^2}_{L^2_z}+ t\log t\norm{\Jbr{z}^{2}|A_1(z)||\nabla^2 S_A|}_{L^2_z} \\
	& \quad + t\norm{\Jbr{z}^{2}|A_1(z)|(|\nabla S_A|^2(\log t)^2+\Jbr{z}^{-2}| S_A|^2+\Jbr{z}^{-1}| \nabla S_A|)}_{L^2_z}\\
	&{}\lesssim_{\rho_0} t\Jbr{\log t}^2\norm{A_1}_{H^{2,2}} \Jbr{\norm{A_1}_{H^{2,2}} }^2 .
\end{align*}
The other terms are easier to handle.
One has similar estimates for $n=1$. Hence, we obtain the result.
\end{proof}

\subsection{Cancellation of the non-resonant part}

\begin{proposition}\label{P:vn2d} 
Let $N_{\mathrm{nr}}$ be the non-resonant part of the nonlinearity given in \eqref{E:resnonres2d}.
Let $v_{\mathrm{ap}}$ be as in \eqref{E:vap2d}.
Then, we have 
\begin{multline*}
\|(\Box+1)v_{\mathrm{ap}}-N_{\mathrm{nr}}(u_{\mathrm{ap}})\|_{L_{x}^{2}}\\
\lesssim t^{-2}(\|A_1\|_{H^{2,2}}+\|B_1\|_{H^{2,2}})^2\Jbr{\|A_1\|_{H^{2,2}}+\|B_1\|_{H^{2,2}}}^{2}.
\end{multline*}
\end{proposition}

\begin{proof}
We apply Lemma \ref{L:hc2D} with $m=1$ to obtain
\begin{align*}
	&(\square_{t,x} + 1)v_{\mathrm{ap}} - N_{\mathrm{nr}}(u_{\mathrm{ap}})\\
	&{}= t^{-2} \sum_{n\neq -1,0}  e^{-i(2n+1)\tau }f_{2,2n+1}(\sinh \sigma \cos \omega, \sinh \sigma, \sin \omega ,\tau) \\
	&\quad +t^{-3} \sum_{n\neq -1,0}  e^{-i(2n+1)\tau }R_{1,2n+1}(\sinh \sigma \cos \omega, \sinh \sigma, \sin \omega ,\tau) \\
	&\quad + t^{-4} \sum_{n\neq -1,0}e^{-i(2n+1)\tau} R_{2,2n+1}(\sinh \sigma \cos \omega, \sinh \sigma, \sin \omega ,\tau)
\end{align*}
for $(t,x)$ such that $t>1$ and $|x|<t$, where
\[\textstyle
	f_{2,2n+1}(z_1,z_2,\tau) = -2(2n+1)i\frac{\partial}{\partial \tau} H_n(z_1,z_2,\tau),
\]
\[\textstyle
	R_{1,2n+1}(z_1,z_2,\tau) = \Jbr{z} \( 2(2n+1)i - 2\frac{\partial}{\partial \tau} + \frac{2}{\tau} \) H_n(z_1,z_2,\tau),
\]
and
\begin{align*}
	R_{2,2n+1}(z_1,z_2,\tau) ={}& \textstyle\Jbr{z}^2 \tau^2 \frac{\partial^2}{\partial \tau^2} H_n(z_1,z_2,\tau)
	 - \Jbr{z}^{3}\Delta_z (\Jbr{z}^{-1}H_n) (z_1,z_2,\tau)\\
	& - \Jbr{z}^{3} \begin{pmatrix}z_1 \\ z_2 \end{pmatrix}^t \nabla_z^2 (\Jbr{z}^{-1}H_n(z_1,z_2,\tau)) \begin{pmatrix}z_1 \\ z_2 \end{pmatrix} +2 H_n (z_1,z_2,\tau)
\end{align*}
with
\begin{align*}
	H_n(z_1,z_2,\tau) ={}&\textstyle\frac{\l}{1-(2n+1)^2} A_1(z)^{\frac{n+3}2} \overline{A_1(z)}^{\frac{-n+1}2} B_1(z)^{-\frac{n}2} \overline{B_1(z)}^{\frac{n}2}\\
	&{}\times L_{n}(\zeta(z))e^{i((n+1)S_A(z)-n S_B(z))(\log \tau +\log \Jbr{z})} .
\end{align*}

For $n\neq 0,-1$, we have
\[
	|f_{2,2n+1}(z_1,z_2,\tau)| \lesssim_{\rho_0} \Jbr{n}^{-3} \tau^{-1} (|S_A(z)|+|S_B(z)|)|A_1(z)|^2,
\]
from which we obtain estimate for $f_{2,2n+1}$:
\[\textstyle
	\norm{\Jbr{z}^{-2}f_{2,2n+1}\(z_1,z_2,\frac{t}{ \Jbr{z}}\) }_{L^2_z}
	\lesssim_{\rho_0} t^{-1}\Jbr{n}^{-3} (\norm{A_1}_{L^6}+\norm{B_1}_{L^6})\norm{A_1}_{L^6}^2.
\]
This is acceptable by Sobolev embedding $H^{2/3} \hookrightarrow L^6$.

Let us next consider $R_{1,2n+1}$.
For $n\neq 0,-1$,
\begin{align*}
	&|R_{1,2n+1}(z_1,z_2,\tau)| \\
	& \lesssim_{\rho_0} \Jbr{n}^{-4}\Jbr{z}|A_1(z)|^2 + \Jbr{n}^{-3}
 \tau^{-1} \Jbr{z}(|S_A(z)|+|S_B(z)|)|A_1(z)|^2 \\
	&\qquad{}+ \Jbr{n}^{-5} \tau^{-1} \Jbr{z}|A_1(z)|^2
\end{align*}
and so
\begin{align*}
	&\textstyle\norm{\Jbr{z}^{-2}R_{1,2n+1}\(z_1,z_2,\frac{t}{ \Jbr{z}}\) }_{L^2_z}\\
	&\lesssim_{\rho_0} \Jbr{n}^{-4}\norm{A_1}_{L^4}^2+ t^{-1}\Jbr{n}^{-3} (\norm{A_1}_{L^6}+\norm{B_1}_{L^6})\norm{A_1}_{L^6}^2
	+ \Jbr{n}^{-5}t^{-1} \norm{A_1}_{L^4}^2.
\end{align*}

The estimate for $R_{2,2n+1}$ is similar to those in the proof of Proposition \ref{P:uapNr2D}.
Each derivative produces multiplication by $O(\Jbr{n})$. It is easy to see if derivatives hit to the phase factor. 
In the case where derivatives hit to the amplitude factor, we will use 
\begin{align*}
	&|\nabla (A_1^{\frac{n+3}2} \overline{A_1}^{\frac{-n+1}2}B_1^{-\frac{n}2} \overline{B_1}^{\frac{n}2} )|\\
	&{}\lesssim \Jbr{n} |A_1||\nabla A_1| + \Jbr{n} |A_1|^2|B_1|^{-1}|\nabla B_1| \\
	&{}\lesssim \Jbr{n}(\zeta^{-2} +1)(|A_1||\nabla A_1|+|B_1||\nabla B_1|)
\end{align*}
and
\begin{align*}
	&|B_1\nabla^2 (A_1^{-\frac{n}2} \overline{A_1}^{\frac{n}2} B_1^{\frac{n+3}2} \overline{B_1}^{\frac{-n+1}2})|\\
	&{}\lesssim \Jbr{n}^2(\zeta^{-2} +1)(|A_1||\nabla^2 A_1|+|\nabla A_1|^2+|B_1||\nabla^2 B_1|+|\nabla B_1|^2).
\end{align*}
Similarly, when the derivative hit to $L_n$, the decay order in $n$ becomes worse,
which also can be understood as the appearance of $O(\Jbr{n})$ factor produced by derivative.
The other difference is that second derivative $\partial_z^2$ is replaced by $\Jbr{z} \partial_z^2 \Jbr{z}^{-1}$.
This creates several lower order terms but they are easier to handle.
In conclusion, the resulting order is no worse than $O(\Jbr{n}^{-3})$ because $A_n,B_n$ is originally $O(\Jbr{n}^{-5})$,
and hence we can sum up with respect to $n$.
\end{proof}

\subsection{Proof of Proposition \ref{P:ua2D}}

\begin{proof}[Proof of Proposition \ref{P:ua2D}]
The first estimate \eqref{eq:ua12d} is immediate by definition of $u_{\mathrm{ap}}$ and $v_{\mathrm{ap}}$.
Let us prove \eqref{eq:ua22d}. By the triangle inequality,
\begin{multline*}
	\norm{(\square +1)\widetilde{u} - N(\widetilde{u})}_{L^2_x}
	\le \norm{(\square +1)u_{\mathrm{ap}} - N_{\mathrm{r}}(u_{\mathrm{ap}})}_{L^2_x}
	\\+ \norm{(\square +1)v_{\mathrm{ap}} - N_{\mathrm{nr}}(u_{\mathrm{ap}})}_{L^2_x}
	+\norm{N(u_{\mathrm{ap}})-N(u_{\mathrm{ap}}+v_{\mathrm{ap}})}_{L^2_x}.
\end{multline*}
The first two terms of the right hand side are handled by Propositions \ref{P:uapNr2D} and \ref{P:vn2d}.
Finally, 
\begin{align*}
	\norm{N(u_{\mathrm{ap}})(t)-N(u_{\mathrm{ap}}+v_{\mathrm{ap}})(t)}_{L^2_x}
	&{}\lesssim (\norm{u_{\mathrm{ap}}(t)}_{L^\I}+ \norm{v_{\mathrm{ap}}(t)}_{L^\I})\norm{v_{\mathrm{ap}}(t)}_{L^2}\\
	&{}\lesssim t^{-2} \norm{A_1}_{L^2}\Jbr{\norm{A_1}_{L^\I} + \norm{B_1}_{L^\I}}^3.
\end{align*}
Combining the above estimates, we complete the proof of \eqref{eq:ua22d}.

Let us show $\tilde{u}(t)$ satisfies \eqref{E:condc2d} if $\{\xi \in \R^2\ |\ |A_1(\xi)|\neq |B_1(\xi)| \}$ has positive measure.
Notice that if we show the set is equal to $\{\xi \in \R^2\ |\ S_A(\xi) + S_B(\xi)\neq 0 \}$ then
the desired conclusion can be obtained by the same argument
as in the proof of Proposition \ref{P:ua}.
One easily sees that $|A_1(\xi)|= |B_1(\xi)|$ implies $S_A(\xi) + S_B(\xi)= 0$.
Hence, it suffices to show that if $S_A(\xi) + S_B(\xi)= 0$ then $|A_1(\xi)|= |B_1(\xi)|$.

Suppose that $S_A(\xi) + S_B(\xi)= 0$. By Assumption \ref{A:2D}, $A_1(\xi)=0$ implies $B_1(\xi)=0$ and so $|A_1(\xi)|=|B_1(\xi)|$.
Consider the case $A_1(\xi)\neq 0$. We have
\begin{equation}\label{E:condc2dpf1}
	0=S_A(\xi) + S_B(\xi) = -\frac{\l}2 (L_0(\zeta(\xi))-\zeta(\xi) L_0(\zeta(\xi)^{-1})) |A_1(\xi)|.
\end{equation}
Remark that
\[
	L_0(\zeta) =\textstyle \frac1{\pi} \int_{0}^{\pi} \sqrt{1+\zeta^2 + 2\zeta \cos \theta} (1 +\zeta \cos \theta)d\theta
\]
and so that
\begin{align*}
	&L_0(\zeta)- \zeta L_0(\tfrac1\zeta) \\
	&{}\textstyle= \tfrac{\zeta^2-1}{\pi\zeta} \int_{0}^{\pi} \sqrt{1+\zeta^2 + 2\zeta \cos \theta} \cos \theta d\theta\\
	&{}\textstyle= \tfrac{4(\zeta^2-1)}{\pi} \int_{0}^{\pi/2} (\sqrt{1+\zeta^2 + 2\zeta \cos \theta}+\sqrt{1+\zeta^2 - 2\zeta \cos \theta})^{-1} \cos^2 \theta d\theta.
\end{align*}
This shows that $L_0(\zeta)-\zeta L_0(\zeta^{-1})=0$ if and only if $\zeta=1$.
Hence, the above identity \eqref{E:condc2dpf1} implies that $\zeta(\xi)=1$, that is, $|A_1(\xi)|=|B_1(\xi)|$.
\end{proof}

\appendix

\section{Properties of $L_0(\zeta)$ in two dimensions}\label{S:L0}

In this section, we give a proof of representation \eqref{E:L0EKform} of $L_0(\zeta)$ defined by \eqref{E:Lndef}, i.e.
\[\textstyle
	L_0(\zeta) = \frac1{2\pi} \int_{0}^{2\pi} |1+\zeta e^{-i\theta}|(1 +\zeta e^{-i\theta})d\theta
\]
for $\zeta>0$.

\subsection{Explicit representation}

We first note that $|1+\zeta e^{-i\theta}|$ is an even function and so
\begin{align*}
	L_0(\zeta) &{}=\textstyle \frac1{\pi} \int_{0}^{\pi} \sqrt{1+\zeta^2 + 2\zeta \cos \theta} (1 +\zeta \cos \theta)d\theta \\
	&{}=\textstyle\frac1{\pi} \int_{0}^{\pi} \sqrt{1+\zeta^2 + 2\zeta \cos \theta}d\theta
	+ \frac{\zeta}{\pi} \int_{0}^{\pi} \sqrt{1+\zeta^2 + 2\zeta \cos \theta} \cos \theta d\theta.
\end{align*}
We estimate each terms in the right hand side. One has
\begin{equation}\label{E:L0pf1}
\textstyle
	 \frac1{\pi} \int_{0}^{\pi} \sqrt{1+\zeta^2 + 2\zeta \cos \theta} d\theta 
	= \frac{2(1+\zeta)}{\pi} E(\frac{2\sqrt{\zeta}}{1+\zeta} ),
\end{equation}
where $E(k)$ is the complete elliptic integral of the second kind.
On the other hand, by integration by parts, we have
\begin{align*}
	&\textstyle\frac{\zeta}{\pi} \int_{0}^{\pi} \sqrt{1+\zeta^2 + 2\zeta \cos \theta} \cos \theta d\theta \\
	&=\textstyle
	\frac{\zeta^2}{\pi} \int_{0}^{\pi} \frac{\sin^2 \theta}{\sqrt{1+\zeta^2 + 2\zeta \cos \theta}} d\theta \\
	&=\textstyle
	\frac{\zeta^2}{\pi\sqrt{1+\zeta^2}}\(\frac{1+\zeta^2}{2\zeta}\)^2 \int_{0}^{\pi} \frac{(1- (\frac{2\zeta}{1+\zeta^2})^2\cos^2 \theta) + ((\frac{2\zeta}{1+\zeta^2})^2-1)}{\sqrt{1 + \frac{2\zeta}{1+\zeta^2} \cos \theta}} d\theta \\
	&=\textstyle\frac{1+\zeta^2}{4\pi} \int_{0}^{\pi} \sqrt{1+\zeta^2 + 2\zeta \cos \theta} d\theta 
	 - \frac{\zeta}{2\pi} \int_{0}^{\pi} \sqrt{1 +\zeta^2 + 2\zeta \cos \theta} \cos \theta d\theta \\
	&{}\quad \textstyle+		\frac{(1+\zeta^2)^{3/2}}{4\pi}((\frac{2\zeta}{1+\zeta^2})^2-1) \int_{0}^{\pi} (1 + \frac{2\zeta}{1+\zeta^2} \cos \theta)^{-\frac12} d\theta .
\end{align*}
Hence,
\begin{equation}\label{E:L0pf2}
\begin{aligned}\textstyle
	\frac{\zeta}{\pi} \int_{0}^{\pi} \sqrt{1+\zeta^2 + 2\zeta \cos \theta} \cos \theta d\theta 
	&=\textstyle\frac{1+\zeta^2}{6\pi} \int_{0}^{\pi} \sqrt{1 +\zeta^2+ 2\zeta \cos \theta} d\theta \\
	&{}\quad\textstyle -\frac{(1+\zeta)^2(1-\zeta)^2}{6\pi \sqrt{1+\zeta^2} } \int_{0}^{\pi} (1 + \frac{2\zeta}{1+\zeta^2} \cos \theta)^{-\frac12} d\theta .
\end{aligned}
\end{equation}
Here, by means of \eqref{E:L0pf1},
\begin{equation}\label{E:L0pf3}\textstyle
	\frac{1+\zeta^2}{6\pi} \int_{0}^{\pi} \sqrt{1 +\zeta^2+ 2\zeta \cos \theta} d\theta
	= \frac{(1+\zeta)(1+\zeta^2)}{3\pi} E(\frac{2\sqrt{\zeta}}{1+\zeta} ).
\end{equation}
Further,
\begin{equation}\label{E:L0pf4}
\begin{aligned}\textstyle
	\int_{0}^{\pi} (1 + \frac{2\zeta}{1+\zeta^2} \cos \theta)^{-\frac12} d\theta 
	&=\textstyle \frac{\sqrt{1+\zeta^2}}{1+\zeta} \int_{0}^{\pi} (1 - \frac{4\zeta}{(1+\zeta)^2} \sin^2 \frac{\theta}2)^{-\frac12} d\theta\\
	&=\textstyle \frac{2\sqrt{1+\zeta^2}}{1+\zeta} K(\frac{2\sqrt{\zeta}}{1+\zeta}),
\end{aligned}
\end{equation}
where $K(k)$ is the complete elliptic integral of the first kind.
Combining \eqref{E:L0pf1}, \eqref{E:L0pf2}, \eqref{E:L0pf3}, and \eqref{E:L0pf4}, we obtain 
the desired formula
\begin{equation}\label{E:L0pf5}
\textstyle
	L_0(\zeta)
=\frac{(1+\zeta)(7+\zeta^2)}{3\pi} E(\frac{2\sqrt{\zeta}}{1+\zeta} )
	- \frac{(1+\zeta)(1-\zeta)^2}{3\pi} K(\frac{2\sqrt{\zeta}}{1+\zeta}).
\end{equation}

\subsection{Singularity at $\zeta=1$}
Let us investigate the regularity of $L_0(\zeta)$.
It is known that $K(k)$ has a logarithmic singularity at $k=1$, that is
\begin{equation}\label{E:Kk1}
	c |\log(1-k)| \le K(k) \le C |\log (1-k)|
\end{equation}
as $k\uparrow1$ (see e.g.\ \cite{Lawden}).
More precise behavior can be found in  \cite{Fuk} and references therein.
We give a short proof of \eqref{E:Kk1}.
To see the upper bound, we employ
\[\textstyle
	K(k) = \int_{0}^1 (1-t^2)^{-\frac12} (1-k^2t^2)^{-\frac12} dt \le \int_{0}^1 (1-t)^{-\frac12} (1-kt)^{-\frac12} dt .
\]
Take $n$ so that $n-1 \le (1-k)^{-1} \le n$ and subdivide the interval of the integration into $I_{m}$, $m=1,2,3,\dots,n-1$
by using the relation $\frac{m}n\le 1-kt \le \frac{m+1}n$.
In each subinterval $I_m$, we use $ (1-kt)^{-1/2} \le (\frac{n}m)^{1/2}$ and integrate $\int_{I_m} (1-t)^{-\frac12}dt$. Then, we obtain a bound
\[\textstyle
	K(k) \le C \sum_{m=1}^{n-1} \frac1m \le C|\log (1-k)|.
\]
The lower bound can be obtained in a similar way.

Further, using the relations $\frac{d}{dk}K(k)=\frac{E(k)}{k(1-k^2)} - \frac{K(k)}k$
and $\frac{d}{dk}E(k) = \frac{E(k)}k - \frac{K(k)}{k}$, one sees from \eqref{E:L0pf5} that
\begin{align*}
L_0'(\zeta)
={}&\textstyle
\frac{(\zeta+1)(\zeta^2+1)}{\pi\zeta} E(\frac{2\sqrt{\zeta}}{1+\zeta} ) - \frac{(\zeta+1)(\zeta-1)^2}{\pi\zeta} K(\frac{2\sqrt{\zeta}}{1+\zeta})\\
L_0''(\zeta)
={}&\textstyle
\frac{(\zeta+1)(2\zeta^2-1)}{\pi\zeta^2} E(\frac{2\sqrt{\zeta}}{1+\zeta} ) - \frac{(\zeta-1)(2\zeta^2+1)}{\pi\zeta^2} K(\frac{2\sqrt{\zeta}}{1+\zeta}) 
\end{align*}
and so that
\[\textstyle
	|L_0'''(\zeta)| \ge \frac{(2\zeta^2+1)}{\pi\zeta^2}| K(\frac{2\sqrt{\zeta}}{1+\zeta})| + O(1)
	\ge c |\log (1-\zeta)| + O(1) \to \I
\]
as $\zeta\to1$.
Thus $L_0(\zeta) \in C^2(\R_+)$ but $L_0(\zeta) \not\in C^3(\R_+)$.
It is easy to see that $L_0(\zeta) \in C^\I (\R_+\setminus\{1\})$.

\begin{figure}[h]
\includegraphics{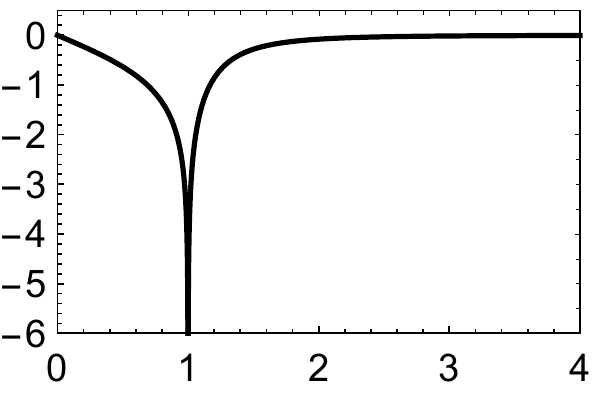}
\caption{A graph of  $L_0'''(\zeta)$}
\end{figure}

\subsection{Asymptotics}
It is well-known that
\begin{align*}
	E(k) ={}& \tfrac\pi2 -\tfrac\pi 8 k^{2} - \tfrac{3\pi}{128}k^4 -\tfrac{5\pi}{512}k^6+ O(k^{8}),\\
	K(k) ={}& \tfrac\pi2 +\tfrac\pi 8 k^{2} + \tfrac{9\pi}{128}k^{4}+\tfrac{25\pi}{512}k^6+ O(k^{8})
\end{align*}
as $k\to0$ (see e.g.\ \cite{Lawden}). By \eqref{E:L0pf5} and \eqref{E:Kk1}, we see that
$L_0(\zeta) \to 1$ as $\zeta\downarrow0$.
Further,
\begin{align*}
	L_0(\zeta) ={}&\textstyle
	\frac{2(1+\zeta)(\zeta+3)}{3\pi}E(\frac{2\sqrt{\zeta}}{1+\zeta} ) + \frac{(1+\zeta)(1-\zeta)^2}{3\pi}(E(\frac{2\sqrt{\zeta}}{1+\zeta} )-K(\frac{2\sqrt{\zeta}}{1+\zeta} ))\\
	={}&\textstyle
	(\frac{2\zeta^2}{3\pi}+\frac{8\zeta}{3\pi}+\frac2\pi)(\tfrac\pi2 -\frac\pi2 \zeta^{-1} + \frac{5\pi}8\zeta^{-2} + O(\zeta^{-3})) \\
	&{}\textstyle + (\frac{\zeta^3}{3\pi}-\frac{\zeta^2}{3\pi}-\frac{\zeta}{3\pi}+O(1))(-\pi \zeta^{-1} +\tfrac{\pi}{2}\zeta^{-2}-\tfrac{3\pi}{4}\zeta^{-3}+ O(\zeta^{-4}) )\\
	={}&\tfrac32\zeta+O(\zeta^{-1})
\end{align*}
as $\zeta\to\I$. In particular, we have $|L_0(\zeta)| \lesssim \Jbr{\zeta}$ for $\zeta >0$. 

\begin{figure}[h]
\includegraphics{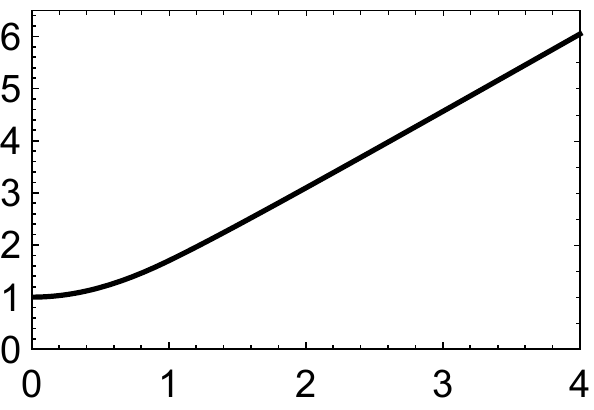}
\caption{A graph of  $L_0(\zeta)$}
\end{figure}

\section{Properties of $L_n(\zeta)$ in two dimensions}\label{S:Lnestpf}

In this appendix, we prove Proposition \ref{P:Lnest}. Namely, we shall show that $L_n(\zeta)$ defined by
\begin{equation}\tag{\ref{E:Lndef}}\textstyle
	L_n(\zeta) = \frac1{2\pi} \int_{0}^{2\pi} |1+\zeta e^{-i\Theta}|(1 +\zeta e^{-i\Theta}) e^{-in\Theta}d\Theta,
\end{equation}
is in $C^1(\R_+) \cap C^\I (\R_+ \setminus \{1\})$ for $n\in \R$, and obeys the bound
\begin{equation}\tag{\ref{E:Lnreg}}
\max_{k=0,1,2}\sup_{n \in \mathbb{Z}} \sup_{\zeta \in [\rho_0^{-1},\rho_0]\setminus\{1\}}\Jbr{n}^{3-k}|L^{(k)}_n(\zeta)| \lesssim_{\rho_0} 1.
\end{equation}
Recall that the identity
\begin{equation}\tag{\ref{E:reflection2D}}
	\zeta^{-2} L_{-n}(\zeta) 
	=   L_{n-1}(1/\zeta)
\end{equation}
is valid. Hence, we may restrict our attention on the case $n\ge0$.

\begin{proof}[Proof of Proposition \ref{P:Lnest}]
We may let $n\neq0$ since the estimate follows from the explicit formula when $n=0$.
Remark that
\begin{align*}
{L}_n(\zeta) 
&=\textstyle \int_0^{2\pi} h(\zeta, \theta)
 e^{-in\theta}d\theta,&
 h(\zeta, \theta) &:= \sqrt{1+\zeta^2+2\zeta\cos \theta}
[\zeta+e^{-i\theta}].
\end{align*}
For each $\zeta \neq1$, $h(\zeta,\cdot)$ is a $2\pi$-periodic smooth function.
Hence, by integration by parts with respect to $\theta$, the desired estimates follow if we show
\[\textstyle
	\sup_{\zeta \in [\rho_0^{-1},\rho_0]\setminus\{1\}}\int_{-\pi}^{\pi} |\partial_\theta^{3-k} \partial_\zeta^{k} h(\theta,\zeta)| d\theta
	\lesssim_{\rho_0} 1
\]
for $k=0,1,2$. 
By symmetry, we may only consider the integrals over $[0,\pi]$ in what follows.
The following estimate will be useful in the estimates below: For $\theta \in (0, \pi)$,
\begin{equation}\label{E:Lnest_key}
\tfrac{1}{\sqrt{1+\zeta^2+2\zeta\cos\theta}} 
= \tfrac{1}{\sqrt{(\zeta+\cos\theta)^2 + \sin^2\theta}}
\le
\min \( \tfrac{1}{|\zeta+\cos\theta|},
\tfrac{1}{|\sin\theta|}\).
\end{equation}
We also note that $|\zeta+e^{ -i\theta}|=\sqrt{1+\zeta^2+2\zeta\cos\theta}$ and $|1+\zeta \cos \theta| = |e^{i\theta}(\zeta+e^{-i\theta}) - i \zeta \sin \theta| \le |\zeta+e^{-i\theta}| + \zeta|\sin \theta|$.

We first consider the case $k=0$.
\begin{align*}
& \partial_\theta^3 h(\zeta, \theta) \\
&{}=  \(\tfrac{\zeta \sin \theta}{\sqrt{1+\zeta^2+2\zeta\cos \theta}}+ \tfrac{-3\zeta^2 \sin \theta \cos \theta}{(1+\zeta^2+2\zeta\cos \theta)^{3/2}}+ \tfrac{-3 \zeta^3 \sin^3 \theta}{(1+\zeta^2+2\zeta\cos \theta)^{5/2}}\) [\zeta+e^{-i\theta}]\\
&\quad - 3  \(\tfrac{-\zeta \cos \theta}{\sqrt{1+\zeta^2+2\zeta\cos \theta}}+ \tfrac{-\zeta^2 \sin^2 \theta}{(1+\zeta^2+2\zeta\cos \theta)^{3/2}}\) ie^{-i\theta}\\
&\quad - 3  \(\tfrac{-\zeta \sin \theta}{\sqrt{1+\zeta^2+2\zeta\cos \theta}}\)e^{-i\theta}
+i \sqrt{1+\zeta^2+2\zeta\cos \theta} e^{-i\theta}\\
&{}=: (I_{0,1}+ I_{0,2} + I_{0,3}) + (I_{0,4} + I_{0,5}) + I_{0,6} + I_{0,7}.
\end{align*}
In light of \eqref{E:Lnest_key}, it holds that $|I_{0,1}|+ |I_{0,6}|+ |I_{0,7}| \le 5\zeta +1$.
We have
\[
	I_{0,4}+I_{0,5} = 3ie^{-i\theta}\zeta\tfrac{(\cos \theta +\zeta)(1+\zeta \cos \theta)}{(1+\zeta^2+2\zeta\cos \theta)^{3/2}}
\]
and
\[
	I_{0,2} + I_{0,3} = -3\zeta^2 \sin \theta (\zeta+e^{-i\theta})\tfrac{(\cos \theta +\zeta)(1+\zeta \cos \theta)}{(1+\zeta^2+2\zeta\cos \theta)^{5/2}}.
\]
Therefore, by using \eqref{E:Lnest_key}, we obtain
\[\textstyle
	\left| \sum_{k=1}^7 I_{0,k} \right| \le 5\zeta +1 + 3\zeta(1+\zeta)\tfrac{|\zeta+\cos \theta|}{1+\zeta^2+2\zeta\cos \theta}.
\]
The desired estimate follows if we show
\begin{equation}\label{E:Lnest_k2}\textstyle
	\int_0^\pi \frac{ |\zeta + \cos \theta|}{1+\zeta^2+2\zeta\cos\theta} d\theta \lesssim_{\rho_0} 1
\end{equation}
for $\rho_0^{-1}\le \zeta \le \rho_0$.
We postpone the proof of the estimate.

We next consider the case $k=1$.
\begin{align*}
 \partial_\theta^2 \partial_\zeta h(\zeta, \theta) 
&{}= - \tfrac{ \cos \theta (1 + \zeta \cos \theta)}{(1+\zeta^2+2\zeta\cos\theta)^{3/2}}[\zeta + e^{-i\theta}]
+ \tfrac{ \zeta \sin^2 \theta }{(1+\zeta^2+2\zeta\cos\theta)^{3/2}}[\zeta + e^{-i\theta}]\\
&\quad - \tfrac{ 3\zeta \sin^2 \theta (1 + \zeta \cos \theta)}{(1+\zeta^2+2\zeta\cos\theta)^{5/2}}
[\zeta + e^{-i\theta}]
 +2i \tfrac{ \sin \theta (1 + \zeta \cos \theta)}{(1+\zeta^2+2\zeta\cos\theta)^{3/2}} e^{-i\theta}\\
&\quad -\tfrac{\zeta + \cos \theta}{\sqrt{1+\zeta^2+2\zeta\cos\theta}}
e^{-i\theta} - \tfrac{\zeta \cos \theta}{\sqrt{1+\zeta^2+2\zeta\cos\theta}}
 - \tfrac{\zeta^2 \sin^2 \theta}{(1+\zeta^2+2\zeta\cos\theta)^{3/2}}\\
&{}=: I_{1,1}+ I_{1,2}+ I_{1,3}+ I_{1,4}+ I_{1,5}+ I_{1,6}+ I_{1,7}.
\end{align*}
By \eqref{E:Lnest_key}, we see that $|I_{1,2}| + |I_{1,5}| \le \zeta + 1$. Remark that
\begin{align*}
	I_{1,1}+I_{1,6}+ I_{1,7} 
	&{}= -\tfrac{ (\zeta+\cos\theta)^2(1+\zeta \cos \theta)}{(1+\zeta^2+2\zeta\cos\theta)^{3/2}}
	+i \tfrac{ \sin\theta \cos \theta (1 + \zeta \cos \theta)}{(1+\zeta^2+2\zeta\cos\theta)^{3/2}}.
\end{align*}
Further,
\[
	|I_{1,3}| \le  \tfrac{3\zeta \sin^2 \theta |1+\zeta \cos \theta|}{(1+\zeta^2+2\zeta \cos \theta)^2}
	\le \tfrac{3\zeta |1+\zeta \cos \theta|}{1+\zeta^2+2\zeta \cos \theta}
\]
and
\[
	|I_{1,4}| \le 2\tfrac{|\zeta+\cos \theta|}{1+\zeta^2 + 2\zeta \cos \theta}.
\]
Thus, combining the above estimates and applying \eqref{E:Lnest_key}, one reach to
\[\textstyle
	\left| \sum_{k=1}^7 I_{1,k}\right| \le 2\zeta +2 + 2\tfrac{|\zeta+\cos \theta|}{1+\zeta^2 + 2\zeta \cos \theta}+ \tfrac{3\zeta +1}\zeta \tfrac{|\zeta^{-1} + \cos \theta|}{1+\zeta^{-2}+2\zeta^{-1} \cos \theta}.
\]
The desired estimate follows from \eqref{E:Lnest_k2}. Remark that the last term is also handled by regarding $\zeta^{-1}$ as $\zeta$.

We finally consider the case $k=2$. We have
\begin{align*}
& \partial_\theta \partial_\zeta^2 h(\zeta, \theta) \\
&{}=  \tfrac{2\sin\theta \cos\theta}{(1+\zeta^2+2\zeta\cos\theta)^{3/2}}
[\zeta + e^{-i\theta}]
+\tfrac{3 \zeta \sin^3\theta }{(1+\zeta^2+2\zeta\cos\theta)^{5/2}}
[\zeta + e^{-i\theta}]
-\tfrac{i\sin^2\theta}{(1+\zeta^2+2\zeta\cos\theta)^{3/2}}
 e^{-i\theta}\\
 &{}\quad - \tfrac{2\sin \theta}{\sqrt{1+\zeta^2+2\zeta\cos\theta}}
+ \tfrac{2(\zeta+\cos\theta)\zeta\sin \theta}{(1+\zeta^2+2\zeta\cos\theta)^{3/2}}\\
&{}=: I_{2,1}+ I_{2,2}+ I_{2,3}+ I_{2,4}+ I_{2,5}.
\end{align*}
We also have
\[
	I_{2,1}+I_{2,3}+I_{2,5} = \tfrac{2\sin\theta(\zeta+ \cos\theta)^2}{(1+\zeta^2+2\zeta\cos\theta)^{3/2}}
-\tfrac{3i\sin^2\theta \cos\theta}{(1+\zeta^2+2\zeta\cos\theta)^{3/2}}- \tfrac{\sin^3\theta}{(1+\zeta^2+2\zeta\cos\theta)^{3/2}}
\]
and
\begin{align*}
	I_{2,2}-\tfrac{3i\sin^2\theta \cos\theta}{(1+\zeta^2+2\zeta\cos\theta)^{3/2}} &{}= 
\tfrac{3 \zeta \sin^3\theta(\zeta + \cos \theta) }{(1+\zeta^2+2\zeta\cos\theta)^{5/2}}
-\tfrac{3i\sin^2\theta( \zeta + \cos \theta)(1+\zeta \cos \theta) }{(1+\zeta^2+2\zeta\cos\theta)^{5/2}}\\
&{}=-\tfrac{3ie^{i\theta}\sin^2\theta( \zeta + \cos \theta)(\zeta +e^{-i\theta}) }{(1+\zeta^2+2\zeta\cos\theta)^{5/2}}.
\end{align*}
In light of \eqref{E:Lnest_key}, we see that $|I_{2,4}|\le 2$ and
\[
	|I_{2,1}+ I_{2,2}+ I_{2,3}+ I_{2,5}|
	\le 3+ \tfrac{3\sin^2\theta |\zeta + \cos \theta|}{(1+\zeta^2+2\zeta\cos\theta)^{2}}.
\]
Hence, using \eqref{E:Lnest_key} again, we have
\[\textstyle
	\left| \sum_{k=1}^5 I_{2,k} \right| \le 5 + 3 \tfrac{|\zeta + \cos \theta|}{1+\zeta^2+2\zeta\cos\theta}.
\]
The desired estimate follows from \eqref{E:Lnest_k2}.

To complete the proof, let us now turn to the proof of \eqref{E:Lnest_k2}.
By changing variable by $t=\tan (\theta/2)$, one has
\[\textstyle
	\int_0^\pi \frac{|1+\zeta \cos \theta| }{1+\zeta^2+2\zeta\cos \theta} d\theta
	= \int_0^\I \frac{|1+\zeta+ (1- \zeta) t^2|}{(1+\zeta)^2 + (1-\zeta)^2 t^2} \frac{2}{1+t^2}dt
\]
We have
\[\textstyle
\int_0^\I \frac{(1+\zeta)}{(1+\zeta)^2 + (1-\zeta)^2 t^2} \frac{2}{1+t^2}dt
\le \frac{2}{1+\zeta} \int_0^\I  \frac{dt}{1+t^2} = \frac{\pi}{1+\zeta}.
\]
On the other hand,
\begin{align*}\textstyle
	\int_0^\I \frac{|(1- \zeta) t^2|}{(1+\zeta)^2 + (1-\zeta)^2 t^2} \frac{2}{1+t^2}dt
	&{}\le \textstyle \int_0^\I \frac{2|1- \zeta|}{(1+\zeta)^2 + (1-\zeta)^2 t^2}dt\\
	&{}= \textstyle\frac{2}{1+\zeta}\int_0^\I \frac{\frac{1+\zeta}{|1- \zeta|}}{(\frac{1+\zeta}{|1-\zeta|})^2 +  t^2}dt
	=\frac{\pi}{1+\zeta}.
\end{align*}
This completes the proof of \eqref{E:Lnest_key}.
\end{proof}

\subsection*{Acknowledgments}
S.M. was supported by JSPS KAKENHI Grant Numbers JP17K14219, JP17H02854, and JP17H02851.
J.S was supported by JSPS KAKENHI Grant Number JP17H02851.


\providecommand{\bysame}{\leavevmode\hbox to3em{\hrulefill}\thinspace}
\providecommand{\MR}{\relax\ifhmode\unskip\space\fi MR }
\providecommand{\MRhref}[2]{%
  \href{http://www.ams.org/mathscinet-getitem?mr=#1}{#2}
}
\providecommand{\href}[2]{#2}

\end{document}